\documentclass[11pt,reqno]{amsart}
\usepackage{xifthen}
\usepackage{subfig}
\usepackage{pkgfile}

\usepackage{amssymb,amsfonts,amsmath,amsthm,amscd,dsfont,mathrsfs}
\usepackage{graphicx,float,psfrag,epsfig}
\usepackage{wrapfig}

\DeclareMathAlphabet{\mathpzc}{OT1}{pzc}{m}{it}

\newtheorem{propo}{Proposition}[section]
\newtheorem{lemma}[propo]{Lemma}
\newtheorem{definition}[propo]{Definition}

\newcommand{\<}{\langle}
\renewcommand{\>}{\rangle}
\newcommand{\sign}{\text{sign}}
\newcommand{\reals}{{\mathds R}}

\newcommand{\eqnsection}{\renewcommand{\theequation}{\thesection.\arabic{equation}}
      \makeatletter \csname @addtoreset\endcsname{equation}{section}\makeatother}

\def\l|{\left|\left|}
\def\r|{\right|\right|}
\def\R{\mathbb R}
\def\P{\mathbb P}

\def\E{\mathbb E}

\def\ve{\varepsilon}

\def\de{{\rm d}}
\def\reals{{\mathds R}}
\def\us{\underline{\sigma}}

\def\R{\mbox{\tiny\rm R}}

\def\Par{{\sf P}}

\def\Pc #1{G ^{\tiny\sf{clon}}(n,p,#1)}

\def\Gr{G_{\tiny {\sf R}}}
\def\Gb{G_{\tiny {\sf B}}}
\def\tGr{\widetilde{G}_{\tiny {\sf R}}}

\def\mcut{{\sf mcut}}

\def\ER{Erd\H{o}s-R\'enyi } 
 
\def\cov{{\rm Cov}}
\def\MaxCut{{\sf MaxCut}} 
\def\Deg{{\mathbf Z}}
\def\Degsec{{\mathbf Y}} 
\def\cen{{\sf{cen}}}

\def\bone{\mathbf{1}}
\def\Parisi{\mathcal{P}}
\def\Xor{{\sf XOR}}
\def\Xorsat{{\sf XORSAT }}
\def\CSP{{\sf CSP}}
\def\bm{\mathbf{m}}
\def\simp{{\sf Sim}}
\def\Pr{{ \sf Pr}}
\def\ba{\mathbf{a} }

\def\RED{{\tiny\sf{RED}}}
\def\BLUE{{\tiny\sf{BLUE}}}

\newtheorem*{lemma*}{Lemma}

\title{Optimization on sparse random hypergraphs and spin glasses}

\author{Subhabrata Sen} 
\address{Microsoft Research New England and MIT Mathematics, \tt{subse@microsoft.com}} 
\thanks{Research partially supported by the William R. and Sara Hart Kimball Stanford Graduate Fellowship}

\date{\today}



\begin{document}

\begin{abstract}
We establish that in the large degree limit, the value of certain optimization problems on sparse random hypergraphs is determined by an appropriate Gaussian optimization problem. This approach was initiated in \cite{dembo2015extremal} for extremal cuts of graphs. The usefulness of this technique is further illustrated by deriving the optimal value for Max $q$-cut on \ER and random regular graphs, Max \Xorsat on \ER hypergraphs, and the min-bisection for the Stochastic Block Model. 
\end{abstract}

\maketitle

\section{Introduction}
The study of combinatorial optimization problems on random instances has a long and rich history--- noteworthy examples include the traveling salesman problem (see \cite{steele1997optimization} and the references therein), shortest path problem \cite{janson1999shortest} , minimum weight spanning tree \cite{beveridge1998spanning, cooper2016spanning, frieze1985spanning,friezethoma2000spanning,janson1995spanning, steele1987spanning}, minimum weight matching \cite{aldous2001assignment,frieze2015assignment,linusson2004assignment,nair2005assignment,wastlund2005assignment,wastlund2009assignment} etc. These problems reveal a number of fascinating characteristics, and have attracted the attention of specialists in statistical physics, probability and computer science. Historically, statistical physicists have often studied these problems using non-rigorous techniques, leading to striking predictions \cite{mezardparisi1987link,MPV,parisi1998conjecture}. Subsequent search for rigorous proofs has directly motivated the development of powerful new tools, thus enriching the subject. 
Some of these problems are algorithmically intractable in the worst case.  The study of random instances has also indirectly inspired new algorithmic breakthroughs for these problems in the average case and has provided a useful benchmark for comparison. 

A special class of optimization problems on graphs comprise finding extreme cuts. These problems are fundamental in combinatorics and theoretical computer science. They are also critical for a number of practical applications \cite{DiazReview,PT} . 
Of particular interest is the \MaxCut\, problem which seeks to partition the vertices of a graph $G=(V,E)$ into $V= V_1 \cup V_2$ such that the number of edges between $V_1$ and $V_2$ is maximized. For random \MaxCut, instances are usually chosen from \ER or random regular graph ensembles. Recall that an \ER random graph $G_n \sim G(n, d/n)$ has $V= [n]$, the edges being added independently with probability $d/n$, whereas a random regular graph $ G_n \sim G^{\R}_n(d)$ is drawn at random from the set of all $d$-regular graphs on $n$ vertices. \cite{dembo2015extremal} studies this problem in the large degree limit, and establishes that for $G_n \sim G(n, d/n)$ or $G_n \sim G^{\R}_n(d)$, with high probability as $n \to \infty$, 
\begin{align}
\frac{\MaxCut(G_n)}{n} = \frac{d}{4} + \Par_* \sqrt{\frac{d}{4}} + o_d(\sqrt{d}), \nonumber
\end{align}
where $\Par_*$ is the ground state energy of the Sherrington-Kirkpatrick model (we refer to \cite{dembo2015extremal} for a definition of the constant $\Par_*$). The first step in the proof is a comparison of the optima 
on sparse \ER graphs with large degrees to that of a Gaussian optimization problem on the complete graph. Subsequently, it is  established that up to lower order corrections in $d$, the $\MaxCut$ has the same behavior on \ER and random regular graphs. In this paper, we generalize the results in \cite{dembo2015extremal} substantially, while simultaneously simplifying the proofs. Before introducing a general framework, we discuss a concrete example which illustrates the usefulness of this approach. 

A \Xor-satisfiability (\Xorsat) problem is specified by the number of variables $n$, and set of $m$ clauses of the form
$x_{i_a(1)} \oplus \cdots \oplus x_{i_a(p)} = b_a$, for $a \in \{1,\cdots, m\}$. Here $b_a \in \{0,1\}$, $\oplus$ denotes mod-2 addition and $\{x_1,\cdots, x_n\}$ is a set of $n$-boolean variables. We consider random instances of the \Xorsat problem, where each sub-collection $\{x_{i_1} , \cdots, x_{i_p} \}$ of variables is included independently with probability 
$d(p-1)!/ n^{p-1}$, and $b_a \sim \dBer(1/2)$ i.i.d. for each equation $a \in \{1, \cdots , m\}$. We say that an instance of the problem is satisfiable if there is an assignment of values to the variables $\{x_1, \cdots , x_n \}$ which satisfies all the equations. Otherwise, the instance is un-satisfiable. This problem has been studied by many authors in the past (see \cite{cuckoo, ibrahimi2015xorsat} and references therein for further motivations and applications). It is well known \cite{dubois2002mandler, cuckoo} that there exists a threshold $d^* = d^*(p)$, independent of $n$, such that for $d<d^*$, a random instance is satisfiable with high probability, while for $d>d^*$, it is unsatisfiable with high probability.

Suppose we fix $d > d^*(p)$ so that we are in the unsatisfiability regime. We wish to determine the maximum proportion of equations which can be typically satisfied in a random \Xorsat instance. We refer to Section \ref{sec:max-xorsat} for  background on this problem. The next lemma makes some progress in this direction. We set $\mathcal{S}(n,p,d)$ to be the maximum number of satisfiable equations in an instance of random \Xorsat with parameters $n,p,d$.

\begin{lemma*}
For $d > d^*(p)$ sufficiently large, with high probability as $n\to \infty$, 
\begin{align}
\frac{\mathcal{S}(n,p,d)}{n} = \frac{d}{2p} +  \frac{\Par_ p}{2} \sqrt{\frac{d }{p}} + o_d(\sqrt{d}). 
\end{align}
\end{lemma*}
\noindent
$\Par_p$ is an explicit constant defined in Section \ref{sec:max-xorsat}. Here and henceforth in the paper, a sequence of random variables $Z_n = o_d(\sqrt{d})$ with high probability if and only if there exists a deterministic function $f(d) = o_d(\sqrt{d})$ such that $\P[ | Z_n | \leq f(d)] \to 1$ as $n \to \infty$. 
%
%
%
%
%
%

The Lemma above follows, in part from a general result which we introduce next. 
Consider a finite alphabet $\mathcal{X} $ and fix a function $f : \mathcal{X}^p \to \reals$ which is symmetric in its arguments, i.e., $f(x_1,\cdots, x_p) = f(x_{\pi(1)} ,\cdots, x_{\pi(p)})$ where $\pi$ is any permutation of $\{1,\cdots, p\}$. Throughout the paper, $\{ (i_1, \cdots, i_p) : 1 \leq i_1 \neq i_2 \neq \cdots \neq i_p \leq n\}$ will denote subsets of $\{1, 2, \cdots, n\}$ with size $p$. We note that in particular, the indices in any fixed subset are all distinct.  

\begin{definition}[Symmetric Arrays]
An array of real numbers $\{a_{i_1, i_2, \cdots, i_p} : 1\leq i_1 \neq \cdots \neq i_p \leq n\}$ is called symmetric if for any permutation $\pi$, $a_{i_1, \cdots, i_p} = a_{\pi(i_1), \cdots, \pi(i_p) }$. A symmetric array of random variables $\{X_{i_1, \cdots, i_p} : 1 \leq i_1 \neq \cdots \neq i_p \leq n \}$ is defined similarly. 
\end{definition}

We are interested in the following optimization problem 
\begin{align}
V_n = \frac{1}{n} \max_{\us \in A_n} \sum_{i_1 \neq \cdots \neq i_p } A_{i_1, \cdots, i_p} f(\sigma_{i_1} ,\cdots , \sigma_{i_p}), \label{eq:opt}
\end{align}
where $A_n \subset \mathcal{X}^n$ and $\{A_{i_1, \cdots, i_p}\}$ is a symmetric array of random variables such that  $\{A_{i_1, \cdots, i_p} : 1 \leq i_1 < i_2 < \cdots < i_p \leq n\}$ are independent and uniformly bounded. \eqref{eq:opt} arises naturally in many contexts--- see Section \ref{section:examples} for concrete applications. Recall that 
a $p$-uniform hypergraph $G=(V,E)$, where $V$ is the set of vertices and the edge set $E$ consists of $p$-subsets of $V$. In particular, a $2$ uniform hypergraph is a graph $G=(V,E)$. 
%
Typically, the variables $A_{i_1, \cdots, i_p}$ represent symmetric (random) weights on the edges of a sparse random hypergraph. 
Consider symmetric, non-negative, bounded kernels $\kappa$, $\kappa_1$ and $\kappa_2$, and a positive constant $d>0$. Throughout, we assume that the kernels are symmetric maps from $\mathbb{N}^p$ to $\reals^+$. We will assume that $|A_{i_1, \cdots , i_p}| \leq B_U$ and

\begin{align}
\P[A_{i_1, \cdots, i_p} \neq 0]  = d \frac{\kappa(i_1 ,\cdots, i_p)}{n^{p-1}}, \,\,\,\, 
\E[A_{i_1,\cdots, i_p}] = d \frac{\kappa_1(i_1, \cdots, i_p)}{n^{p-1}}, \,\,\,\,
\E[A_{i_1, \cdots, i_p}^2] = d \frac{ \kappa_2(i_1, \cdots , i_p) }{n^{p-1}}. \nonumber 
\end{align}
In most of our applications, the kernels $\kappa$,  $\kappa_1$ and $\kappa_2$ are constants.
The parameter $d$ is intrinsically related to the degree of a vertex in the hypergraph.  We are specifically interested in the case when the parameter $d$ is a large constant independent of $n$. Consider also the following ``gaussian" optimization problem.
\begin{align}
T_n^{\alpha} = \frac{1}{n} \max_{\us \in A_n} \sum_{i_1 \neq \cdots \neq i_p} \frac{J_{i_1,\cdots, i_p}}{n^{(p-1)/2}} f(\sigma_{i_1}, \cdots, \sigma_{i_p}).  \label{eq:gaussian_opt} \\ 
{\textrm{subject to}}\,\,\, \frac{1}{n^p} \sum_{i_1 \neq \cdots \neq i_p} \kappa_1(i_1, \cdots, i_p) f(\sigma_{i_1}, \cdots, \sigma_{i_p}) =\alpha. \nonumber
\end{align}
$\{J_{i_1, \cdots, i_p} \}$ is a symmetric array of random variables such that $\{ J_{i_1,\cdots, i_p}: 1\leq i_1 <  i_2 < \cdots < i_p  \leq n \}$ are independent $\mathcal{N}(0, \kappa_2(i_1, \cdots, i_p))$. Our main result relates these two optimal values, up to a small error when $d$ is large. 

\begin{thm}
\label{lemma:comparison}
For $d$ sufficiently large, with high probability as $n\to \infty$, 
\begin{align}
V_n = \E[\max_{\alpha} ( \alpha d + T_n^{\alpha} \sqrt{d} )] + o_d(\sqrt{d}) . \label{eq:relation_value}
\end{align}
\end{thm}
\noindent
The maximum over an empty set should be interpreted to be $-\infty$ in the statement of Theorem \ref{lemma:comparison}. The event $A$ occurs with high probability(w.h.p.) if $\P[A] \to 1$ as the problem size $n \to \infty$. Theorem \ref{lemma:comparison} relates the value of the optimization problem \eqref{eq:opt} to its gaussian analogue \eqref{eq:gaussian_opt}, which is often more tractable and thus furnishes us with a powerful general method to systematically study a large class of optimization problems, thereby generalizing the first part of the argument in \cite{dembo2015extremal}. Indeed, this principle is extremely robust and independent of the function $f$ in the objective.  From the statistical physics viewpoint, Theorem \ref{lemma:comparison} rigorously establishes a connection between the optimal value of problems on sparse graphs to the ground state of disordered spin-glass models. Even in cases where the ground state energy cannot be rigorously evaluated, one can hope to use probabilistic tools to derive bounds on these quantities, possibly improving on the bounds derived by purely combinatorial techniques. 


For a $p$-uniform hypegraph $G$, we define the ``adjacency matrix" $\{A(i_1, \cdots, i_p) : 1\leq i_1 \neq \cdots \neq i_p \leq n \}$ such that 
\begin{align}
A(i_1, \cdots , i_p ) = 
\begin{cases}
\frac{1}{(p-1)!} & \mbox{if} \hspace{5pt} \{i_1, \cdots, i_p\} \in E, \\
 0 & \mbox{o.w.} 
 \end{cases}\nonumber 
\end{align}
The optimization problems \eqref{eq:opt} have been typically studied on sparse random \ER hypergraphs where $A_{i_1,\cdots, i_p}$ denotes the adjacency matrix of the hypergraph. This special case is recovered by setting $\kappa_1 = 1, \kappa_2 =1/ (p-1)!$ in the setup above. Throughout the rest of this section, whenever we refer to \eqref{eq:opt}, we will assume implicitly this specific choice of the kernels. 
Another class of random instances which are of natural interest comprise the corresponding optimization problems on random regular hypergraphs. Let $G_n(p,d)= ([n], E)$ denote a random $p$-uniform $d$-regular hypergraph on $n$-vertices. As usual, the degree of a vertex $v \in [n]$ denotes the number of hyper-edges $e\in E$ with $v \in e$. We implicitly assume that $p | nd$ and note that $G^{\R}(n,d) = G_n(2,d)$. As above, we seek the optimum value \eqref{eq:opt} and denote this value by $V_n^{\R}$. Our next result derives a Gaussian surrogate for this value.  To this end, define 
\begin{align}
S_n^{\alpha} = \frac{1}{n} \max_{\us \in A_n}&\Big[ \sum_{i_1, \cdots , i_p} \frac{J_{i_1,\cdots, i_p}}{n^{(p-1)/2}} f(\sigma_{i_1}, \cdots, \sigma_{i_p}) - \frac{p}{n^{p-1}} \sum_{i_1,\cdots , i_p} G_{i_1} f(\sigma_{i_1}, \sigma_{i_2} ,\cdots, \sigma_{i_p}) \Big].  \label{eq:regular_opt} \\ 
&{\textrm{subject to}}\,\,\, \frac{1}{n^p} \sum_{i_1 , \cdots , i_p} f(\sigma_{i_1}, \cdots, \sigma_{i_p}) =\alpha , \nonumber
\end{align}
where $\{J_{i_1, \cdots, i_p}\}$ is a symmetric array of random variables defined as follows. Let $\{V_{i_1, \cdots, i_p} : 1\leq i_1 ,\cdots, i_p \leq n\}$ be an array of i.i.d $\dN(0,1)$ random variables and we define 
$
J_{i_1, \cdots, i_p} = \frac{\sqrt{p}}{{p!}} \, \sum_{\pi} V_{\pi(i_1) ,\cdots, \pi(i_p)}, 
$
where $\pi$ is a permutation of $\{1, \cdots, p\}$. This array will be referred to as a ``standard symmetric $p$-tensor" of Gaussian variables.  
 Further, $G_i = \sum_{m_2 , \cdots, m_p} \frac{J_{i, m_2 , \cdots , m_p } }{n^{(p-1)/2}}$. 

\begin{thm}
\label{thm:regular}
For $d$ sufficiently large, with high probability as $n\to \infty$, we have,
\begin{align}
V_n^{\R} = \E \max_{\alpha} [ d \alpha + \sqrt{d} S_n^{\alpha} ] + o_d(\sqrt{d}). 
\end{align}
\end{thm}
\noindent
\cite{dembo2015extremal} deduced a similar Gaussian surrogate for the \MaxCut \, on random $d$-regular graphs for large $d$. However, the proof is tailor made for the \MaxCut\, and Theorem \ref{thm:regular} considerably generalizes this result. 
It is interesting to note that the Gaussian surrogate for the optimization problem on the sparse \ER random hypergraphs \eqref{eq:relation_value}  and $d$-regular random hypergraphs \eqref{eq:regular_opt} are not always the same. However, the next proposition derives sufficient conditions which ensure that the optimal values are the same, up to lower order correction terms in $d$. Moreover, a priori, it is unclear how to optimize \eqref{eq:relation_value}. We observe that the objective in {\tiny\textbf{\sf RHS}} of \eqref{eq:relation_value} consists of two terms--- a deterministic term, which is in general of order $d$, and the random contribution, which is in general of order $\sqrt{d}$. If $d$ is large, one expects that the optimum would be attained by first constraining $\us$ to maximize the deterministic term, and then maximizing the random term subject to this constraint. We can actually formalize this idea under suitable assumptions on the function $f$. 
 
 To this end, we introduce some notation. For each configuration $\us \in \mathcal{X}^n$, we define the vector $\bm(\us)= (m_k(\us) : k \in \mathcal{X})$, where $m_k(\us) = \sum_{i=1}^{n} \bone(\sigma_i = k)/n$ denotes the proportion of spins in $\us$ which are of type $k$. We denote the $|\mathcal{X}|-1$ dimensional simplex by $\simp = \{\mathbf{a} = (a_1, \cdots, a_{|\mathcal{X}|}): a_i \geq 0, \sum_i a_i =1 \}$. Now, given any function $f: \mathcal{X}^p \to \reals$, we define the smooth function $\Psi : \simp \to \reals$, 
 \begin{align}
 \Psi (\ba) = \sum_{j_1, \cdots , j_p \in \mathcal{X}} f(j_1, \cdots, j_p) a_{j_1} \cdots a_{j_p}. \label{eq:defnpsi}
 \end{align}
 
  Note that $\Psi$ may also be thought of as  a function of the independent variables $(a_1, \cdots, a_{|\mathcal{X}|-1} )$ and in this case, we shall denote the function as $\bar{\Psi}$. This will allow us to exploit smoothness properties of $\Psi$. Finally,
 consider the following criteria.
\begin{itemize}
\item[(C1)]  For all $\us \in A_n$, $n\geq 1$, there exists some constant $\eta$ such that for all $j \in \mathcal{X}$, 
\begin{align}
\sum_{j_2, \cdots, j_p} f(j, j_2, \cdots, j_p) m_{j_2}(\us) \cdots m_{j_p}(\us) = \eta  + r_j(\us), \nonumber
 \end{align}
 where the residual vector $r(\cdot) = (r_j (\cdot): j \in \mathcal{X})$ satisfies $\sup_{\us \in A_n} \| r(\us) \|_{\infty} = o_d(1)$. 

\item[(C2)] Assume $A_n = \mathcal{X}^n$, $\Psi(\bm)$ is maximized at some $\bm^* = (m_i^*: i \in \mathcal{X}) \in \simp$ such that $m_i^* >0$ for all $i \in \mathcal{X}$. Further, we assume that $-\grad^2 \bar{\Psi} (\bm) \succeq c \mathbf{I}$  for some constant $c>0$ in a neighborhood of $\bm^*$.  
\end{itemize}

Recall that the optimum value $V_n$ for \ER hypergraphs \eqref{eq:opt} corresponds to the special choice $\kappa_1 =1$ and $\kappa_2 = \frac{1}{(p-1)!}$. The following result compares this optimum to that on regular hypergraphs.

\begin{propo}
\label{prop:equivalence}
Under the conditions (C1) or (C2), $V_n - V_n^{\R} = o_d(\sqrt{d})$. Moreover, under the condition $(C2)$, for $d$ sufficiently large, 
\begin{align}
V_n  = d \Psi(\bm^*) + \sqrt{d}\,\E[T_n^{\Psi(\bm^*)}] + o_d(\sqrt{d}). \nonumber 
\end{align}
\end{propo}

\begin{remark}
We note that Proposition \ref{prop:equivalence} along with \cite[Lemma 2.4]{dembo2015extremal} re-derives the main results in 
\cite{dembo2015extremal} in a relatively straightforward manner. While it is arguable that the proof of Proposition \ref{prop:equivalence} is similar to those in \cite{dembo2015extremal}, we believe that this proposition is conceptually simpler, 
and clearly illustrates the basic principles at work.  
\end{remark}

The rest of the paper is structured as follows. Section \ref{section:examples} covers certain applications of Proposition \ref{prop:equivalence} and reports some follow up work. Theorem \ref{lemma:comparison} is proved in Section \ref{section:main-proof} while the proof of Theorem \ref{thm:regular} is in Section \ref{section:proof-regular}. Finally, we prove Proposition \ref{prop:equivalence} in Section \ref{section:proof-prop}. Throughout the paper, $C, C_0, C_1$ are used to denote universal constants independent of the problem size $n$. These constants may change from line to line.

\section{Examples}
\label{section:examples}

In this section, we illustrate the usefulness of Proposition \ref{prop:equivalence} by deriving the optimal value
in some examples. We study the random max XORSAT problem in Section \ref{sec:max-xorsat}, the max $q$-cut of sparse random graphs in Section \ref{sec:max-qcut}, and the minimum bisection of the stochastic block model in Section \ref{sec:sbm-mincut}. Some avenues for future research are discussed in Section \ref{section:future}. The proofs of the results are deferred to Section \ref{section:applications-proof}.
%
%
\subsection{Max XORSAT on \ER random hypergraphs}
\label{sec:max-xorsat}

An instance of the Boolean $k$-SAT problem consists of a boolean {{\sf CNF}} formula ({{\sf AND}} of {{\sf OR}}s) in $n$-variables with $k$-literals per clause. The decision version of this problem seeks to find an assignment of values to variables such that all clauses evaluate to {{\sf TRUE}}. 
It is a canonical NP hard problem arising from theoretical computer science \cite{karp1972hardness}.  The search for hard instances has motivated the study of random $k$-SAT problems. This area has witnessed an explosion of activity in the last decade due to confluence of ideas from computer science, statistical physics and mathematics. Statistical Physicists anticipate certain intriguing attributes in random instances of these problems based on predictions from non-rigorous replica and cavity methods (see \cite{krzakala2007gibbs} and references therein). It was conjectured that in random $k$-SAT and a wide variety of  general Constraint Satisfaction Problems (\CSP s), there is a sharp phase transition as the number of constraints grows--- a ``typical" problem is satisfiable with high probability below this threshold while above this threshold, it is unsatisfiable with high probability.  
It is an outstanding mathematical challenge to rigorously establish these predictions. Recently, significant progress has been achieved in this regard \cite{dss_1,dss_2, dss_3, sly_sun_zhang_16}. 

\cite{Coppersmith} initiated a study of the natural optimization version of the $k$-SAT, where one wishes to maximize the number of satisfied clauses. They established a phase transition for this problem (see \cite{Coppersmith} for the exact random ensemble used)--- below the satisfiability threshold all the clauses are satisfied, while above this threshold, the minimum number of unsatisfiable clauses is $\Theta(n)$.
 A natural question in this direction is to determine the maximal number of satisfiable clauses in a random \CSP\, above the satisfiability threshold. In this paper, we answer this question for the random {{\sf XORSAT}} problem, in the large degree limit.

To state our result, we need to introduce some notation from \cite{Pan}. On the binary hypercube $\{-1, +1\}^n$, fix $\beta>0, h\in \reals$ and consider the Gaussian process $H_n(\us) =\beta H_n'(\us) + \beta h \sum_i \sigma_i$, where 
\begin{align}
H_n'(\us) = \sum_{i_1, \cdots, i_p} \frac{G_{i_1, \cdots, i_p}}{n^{(p-1)/2}} \sigma_{i_1}\cdots \sigma_{i_p}, \nonumber
\end{align} 
with $\{G_{i_1,\cdots, i_p} : 1 \leq i_1, \cdots , i_p \leq n\}$  an array of i.i.d. $\dN(0,1)$ random variables. $H_n'$
is a centered Gaussian process with covariance 
\begin{align}
\cov(H_n'( \us) , H_n'(\us')) = n \beta^2 \Big(\frac{\langle \us, \us' \rangle}{n} \Big)^p. \nonumber
\end{align}
The process $\{H_n(\us): \us \in \{\pm 1\}^n \}$ is usually referred to as the $p$-spin model in the spin glass literature. It was conjectured by Parisi and later proved by Talagrand \cite{TalagrandFormula} and Panchenko \cite{panchenko2014mixed}  that the following limit exists with probability $1$.  
\begin{align}
F_p (\beta, h) := \lim_{n \to \infty} \frac{1}{n} \log \sum_{\us \in \{\pm 1\}^n } \exp(H_n(\us)) = \inf_{\mu \in \Pr([0,1])} \mathcal{P}(\mu; \beta, h) . \label{eq:Parisi}
\end{align}
Here $\Pr([0,1])$ denotes the space of all probability measures on $[0,1]$ and $\mathcal{P}(\cdot; \beta, h)$ is the Parisi functional, defined as 
\begin{align}
\mathcal{P}(\mu; \beta,h) = u_{\mu}(0, \beta h) - \frac{p(p-1)\beta^2 }{2} \int_{0}^{1} t^{p-1} \mu([0,t]) \de t, \nonumber 
\end{align}
where $u_{\mu}(\cdot, \cdot)$ solves the Parisi {\tiny{\sf PDE}} 
\begin{align}
&\partial_t u_{\mu}(t,x) + \frac{p(p-1)\beta^2}{2} \Big( \partial_{xx} u_{\mu}(t,x) + \mu([0,t]) (\partial_x u_{\mu}(t,x))^2 \Big) =0 \,\,\,\,\,\, (t,x) \in (0,1) \times \reals, \nonumber\\
&u_{\mu}(1,x) = \log \cosh(x). \nonumber 
\end{align}
See \cite{jagannath2015dynamic} and references therein for regularity properties of the {{\sf PDE}} and uniqueness of solutions. 
Using continuity in $\beta$, it is easy to see that $\lim_{\beta \to \infty} F_p(\beta, h) / \beta$ exists and we define
$\Par_p = \lim_{\beta \to \infty} \frac{F_p(\beta,0)}{\beta}$.
Recently, Auffinger and Chen \cite{auffinger2016groundstate} have derived a direct `zero temperature' variational representation for $\Par_p$. We refer the reader to \cite{auffinger2016groundstate} for further details.  
Having introduced this notation, we re-state our result for the convenience of the reader. Recall the satisfiability threshold $d^*(p)$ introduced above. 

\begin{lemma}
\label{lemma:xorsat}
For $d > d^*(p)$ sufficiently large, with high probability as $n\to \infty$, 
\begin{align}
\frac{\mathcal{S}(n,p,d)}{n} = \frac{d}{2p} +  \frac{\Par_ p}{2} \sqrt{\frac{d }{p}} + o_d(\sqrt{d}). 
\end{align}
\end{lemma}

\subsection{Max $q$-cut in \ER and random regular graphs}
\label{sec:max-qcut}
Consider the random graph $G(n, d/n)$. 
These graphs are typically sparse with $O(n)$ edges and the degree of a typical vertex concentrates around $d$. The Max $q$-cut problem seeks to divide the vertices into $q$-groups such that the number of edges among the parts is maximized. The case $q=2$ is usually referred to as the \MaxCut\, problem on graphs. This class of problems has attracted significant attention from the Theoretical Computer Science and the Mathematics communities. We refer to \cite{dembo2015extremal} for a survey on the existing results for Maxcuts on sparse \ER or random $d$-regular graphs. The $q$-cut has received less attention. \cite{cojaoghlan2003maxkcut} studied the Max $q$-cut and the performance of the Goemans-Williamson SDP relaxation on sparse \ER random graphs. For an \ER random graph $G_n \sim G(n, d/n)$, they prove that there exist constants $0<C_1< C_2$ such that with high probability as $n \to \infty$,  
\begin{align}
 \frac{d}{2}\Big(1 - \frac{1}{q} \Big) + C_1 \sqrt{d}\leq \frac{\MaxCut(G_n,q)}{n} \leq \frac{d}{2}\Big(1 - \frac{1}{q} \Big) + C_2 \sqrt{d}. \nonumber 
\end{align} 

This problem may be formulated as in \eqref{eq:opt} as follows. Each $q$-partition may be encoded by an assignment of ``spin" variables $\sigma_i \in \mathcal{X}= [q]=\{1,\cdots, q\}$, which indicates the membership of the vertex to one of the groups in the partition. In statistical physics terminology, the value of Max $q$-Cut is closely related to the ground state of the Antiferromagnetic Potts model. The next result specifies the asymptotic value of the Max $q$-cut on sparse \ER random graphs.
 To state the result, we introduce some notation from \cite{panchenko2015potts}. We define
\begin{align}
 \Gamma_q &= \{ \gamma : \gamma \textrm{ is a } (q-1) \times (q-1) \textrm{ symmetric positive semidefinite matrix} \}, \nonumber\\
 \Pi &= \{ \pi : [0,1] \to \Gamma_q : \pi  \textrm{ is left continuous }, \pi(x) \leq \pi(x') \textrm{ for } x \leq x' \}, \nonumber  
\end{align}
where $\pi(x) \leq \pi(x')$ means that $\pi(x') - \pi(x) \in \Gamma_q$. Next, we define $\Pi_q = \{ \pi \in \Pi : \pi(0)=0, \pi(1) = \textrm{diag}(1/q , \cdots, 1/q) \}$. Given $\gamma \in \Gamma_q$, we define an expansion map $U: \Gamma_q \to \Gamma_{q+1}$ as follows. For $1\leq k , k' \leq (q-1)$, $U(\gamma)_{k k'} = \gamma_{k.k'}$. Otherwise, we set, for $1\leq k \leq (q-1)$, 
\begin{align}
U(\gamma)_{q,k} = U(\gamma)_{k,q} = \frac{1}{q} - \sum_{k'} \gamma_{k,k'} , \,\,\,,\,
U(\gamma)_{q,q} = - \Big(1- \frac{2}{q} \Big) + \sum_{k, k'} \gamma_{k k'}. \nonumber 
\end{align}
It is easy to note that $U$ is non-decreasing and thus for any $\pi \in \Pi$,  $U(\pi) : [0,1] \to \Gamma_q$ is left-continuous and non-decreasing. 

Armed with the notation introduced above, we define the ``Parisi functional" in this context. 
For some $r \geq 1$, consider two sequences $x_{-1} = 0 \leq x_0  \leq \cdots \leq x_r =1$ and a monotone sequence in $\Gamma_q$, $0 = \gamma_0 \leq \cdots \leq \gamma_r = \textrm{diag}(1/q , \cdots , 1/q)$. We can associate to any such pair a discrete path in $\Pi_q$ such that 
\begin{align}
\pi(x) = \gamma_v, \,\,\, x_{v-1} < x \leq x_v
\end{align}
for $0\leq v \leq r$ with $\pi(0) = 0$. Given such a discrete path, we consider a sequence of independent Gaussian vectors $z_v = (z_v(k))_{k \leq q}$ for $0\leq v \leq r$ with covariances $\cov(z_v) = 2 ( U(\gamma_v) - U(\gamma_{v-1}))$. For any $\lambda \in \reals^{q}$ and $\beta \geq 0$, we set 
\begin{align}
X_r = \log \sum_{k \leq q} \exp \Big( \beta \sum_{v=1}^{r} z_v(k) + \sum_{k' \leq (q-1)} \lambda_{k'} \bone(k=k') \Big). 
\end{align}
Recursively, for $0\leq v \leq r-1$, we define 
\begin{align}
X_v = \frac{1}{x_v} \log \E_v \exp (x_v X_{v+1} ), 
\end{align}
where $\E_v[\cdot]$ denotes the expectation with respect to $z_{v+1}$. If $x_v=0$, we set $X_v= \E_v[X_{v+1}]$. Noting $X_0$ is non-random, we set $\Phi(\beta, \lambda, r, x, \gamma) := X_0$. Finally, we define the Parisi functional 
\begin{align}
\Parisi(\beta,\lambda, r, x, \gamma) = \Phi(\beta,\lambda, r, x , \gamma) - \frac{1}{q} \sum_{k=1}^{ q-1}\lambda_k - \frac{\beta^2}{2} \sum_{v=0}^{r-1} x_v \Big( \| U(\gamma_{v+1}) \|_{\textrm{HS}}^2 - \| U(\gamma_v) \|_{\textrm{HS}}^2 \Big),
\end{align}
where $\| \cdot \|_{\textrm{HS}}$ denotes the Hilbert-Schmidt norm of a matrix. It is easy to see \cite{Pan, panchenko2015potts}
that $\Par_*(q) :=\lim_{\beta \to \infty} \frac{1}{\beta} \inf_{r, \lambda, x ,\gamma} \Parisi(\beta,\lambda, r, x, \gamma) $ exists. 
This allows us to state the following Lemma for this example. 

\begin{lemma}
\label{lemma:Maxqcut}
If $G_n \sim G(n, d/n)$ or $G_n \sim G^{\R}(n,d)$ then
as $n\to \infty$, for $d$ sufficiently large,
\begin{align}
\frac{\MaxCut(G_n,q)}{n} = \frac{d}{2}\Big(1-\frac{1}{q} \Big) + \Par_*(q) \frac{\sqrt{d}}{2} + o_d(\sqrt{d}). 
\end{align}
\end{lemma}


\subsection{Min bisection in the stochastic block model}
\label{sec:sbm-mincut}

The planted bisection model or stochastic block model has been extensively studied in computer science to determine the average case behavior of graph bisection heuristics. Given a fixed subset $S \subset [n]= \{1, \cdots, n\}$ with $|S| = n/2$ (we will assume throughout that $n$ is even) and $a>b>0$, the random graph $G(n, a/n , b/n) = ([n], E_n)$  has a vertex set $[n]$, and edges are added independently with 
\begin{align}
\P[(i,j) \in E_n] = \begin{cases}
a/n & \mbox{ if $\{i,j\}\subseteq S$ or $\{i,j\}\subseteq S^c$,}\\
b/n & \mbox{ if $i\in S,\, j\in S^c$ or $i\in S^c ,j\in S$.}
\end{cases}
\end{align}

This model has also been studied extensively in the statistics literature as a testbed for estimation strategies which recover the true community assignments. Recently, the model attracted intense study due some outstanding conjectures by physicists \cite{decelle2011asymptotic}. These conjectures have been established due the efforts of Mossel, Neeman, Sly \cite{mossel2013proof} and Massoulie \cite{massoulie2014community}. The last few years has witnessed frantic activity in this research area and thus instead of attempting to survey the existing literature, we will point the reader to the excellent survey 
in \cite{guedon2015community} for an overview of the existing results. 

In a different direction, CS studies about the performance of bisection algorithms on the planted bisection model have established that for $(a-b)$ large, the planted bisection is the minimal bisection--- however, for $(a-b)$ small, the planted bisection ceases to be the minimum bisection. This leaves open the basic question: 

\textit{What is the bisection width of a graph drawn from the planted bisection model? }

\noindent This question was partially answered by Coja Oghlan et. al. \cite{coja-oghlan_minimumbisection} who established that for $(a-b)$ sufficiently large, the problem can be solved using a local ``warning propagation" algorithm originally introduced in the study of random \CSP s. 
It turns out to be more natural to parametrize this model in terms of the average degree $d= (a+b)/2$ and the Signal-to-Noise-Ratio  (SNR) parameter $\xi = (a-b)/ \sqrt{2(a+b)} >0$. For example, given the graph, ``non-trivial" inference about the correct community memberships is possible if and only if $\xi >1$ \cite{mossel2013proof,massoulie2014community}. The next lemma estimates the bisection width of the sparse block model in the regime when the average degree $d$ is large, while the SNR parameter $\xi$ is of a constant order. To state our result, we again need to introduce some quantities relevant in this context. It was established in \cite{chen2014mixedferromagnetic} that for a {{ \sf GOE}} matrix $J= (J_{ij})_{1\leq i,j \leq n}$ and $\xi >0$
\begin{align}
\tilde{F}(\beta) := \lim_{n \to \infty} \frac{1}{n} \log \sum_{\us \in \{\pm \}^n} \exp \Big( \beta \Big( \frac{\xi}{n} \langle \bone, \us \rangle^2 + \sum_{i,j} \frac{J_{ij}}{\sqrt{n}} \sigma_i \sigma_j \Big)  \Big) \nonumber
\end{align}
exists and that $\tilde{F}(\beta) = \max_{\mu \in [-1,1]} \Big( F_2( \sqrt{2}\beta, \frac{\xi}{\sqrt{2}}) - \beta \mu^2 \Big)$, where $F_2$ is defined in \eqref{eq:Parisi}. It is easy to see that $\lim_{\beta \to \infty} \tilde{F}(\beta)/ \beta$ exists and we set
$C^* = \lim_{\beta \to \infty} \frac{\tilde{F}(\beta)}{\beta}$.  Denoting the minimum bisection of a graph $G$ as $\mcut(G)$, we have the following result. 

\begin{lemma}
\label{lemma:sbm}
Let $G_n\sim G(n, a/n, b/n)$. Assume that $d$ is sufficiently large. Then we have, as $n \to \infty$, 
\begin{align}
\frac{1}{n}\mcut(G_n) = \frac{d}{4} - C^*\sqrt{d} + o(\sqrt{d}). 
\end{align}
\end{lemma}

\begin{remark} 
We note that a similar strategy has been used in \cite{montanari2015semidefinite} to analyze the performance of some semidefinite programs in this context.  
\end{remark}

\subsection{Follow up work}
\label{section:future}
%
Since the submission of the initial draft, some subsequent papers have used this framework and similar ideas to study combinatorial problems. We take this opportunity to briefly review these new results. In a joint work of Aukosh Jagannath, Justin Ko and the author \cite{ko2017potts}, this framework is crucially used to study the MAX $q$-cut on inhomogeneous random graphs. Similarly, in a joint work of Aukosh Jagannath and the author\cite{jagannath2017unbalanced}, these results are crucial for establishing sharp comparison inequalities between unbalanced cuts on sparse \ER and random regular graphs.


\cite{panchenko2016sat} studies the value of the random MAX $k$-SAT problem in the large degree limit using some related ideas.  The MAX $3$-SAT problem had been studied earlier in this setup by \cite{leuzzi2001sat} using non-rigorous replica methods. 
%
 \cite{chen2016disorder} uses similar ideas en route to establishing disorder chaos in some diluted models. As a consequence, one can establish the proliferation of the near optimizers in these problems. \cite{chen2017suboptimality} use some associated ideas to establish the sub-optimality  of a class of local algorithms on a class of hypergraph MAX-CUT type problems.


\section{Proofs of Examples}
\label{section:applications-proof}

\textit{Proof of \ref{lemma:xorsat}:}
We first express this problem in the setup of \eqref{eq:opt}. To this end, we note that each variable $x_i$, $1\leq i \leq n$,  may be equivalently encoded by a spin variable $\sigma_i$ taking values in the finite alphabet $\mathcal{X}= \{-1, +1\}$. The total number of equations will be denoted by $m$. We also encode each $b_a$, $a\in \{1,\cdots, m\}$ to take values in $\{-1, +1\}$. Thus $\P[b_a = -1] = \P[b_a = +1] = 1/2$ for each $a \in \{1, \cdots , m \}$. We set $f(\sigma_1, \cdots, \sigma_p) = \sigma_1 \cdots \sigma_p$ and we note that the maximum number of satisfiable clauses may be expressed as 
\begin{align}
\mathcal{S}(n,p, d) = \frac{m}{2} +  \frac{1}{2p} \max_{\us \in \mathcal{X}^n}\sum_{i_1 \neq \cdots \neq i_p} A_{i_1, \cdots, i_p} b_{i_1, \cdots, i_p} f(\sigma_{i_1} , \cdots, \sigma_{i_p}), \nonumber 
\end{align}
where $A$ is the adjacency matrix of the corresponding $p$-uniform hypergraph. Therefore, this optimization problem is covered by the framework 
introduced in \eqref{eq:opt}. In this case, we have, $\kappa \equiv (p-1)!$, $\kappa_1 \equiv 0$ and $\kappa_2 =1/(p-1)!$. Applying Theorem \ref{lemma:comparison}, we have, with high probability as $n\to \infty$, 
\begin{align}
\frac{\mathcal{S}(n,p,d)}{n} = \frac{d}{2p} +  \frac{\sqrt{d}}{2np} \E\Big[ \max_{\us \in \mathcal{X}^n} \sum_{i_1 \neq \cdots \neq i_p} \frac{J_{i_1,\cdots, i_p}}{n^{(p-1)/2}} \sigma_{i_1}\cdots \sigma_{i_p} \Big],
\end{align} 
where $\{J_{i_1, \cdots, i_p} \}$ is a standard symmetric Gaussian $p$-tensor. Now, $J_{i_1,\cdots,i_p} = \frac{\sqrt{p}}{p!} \sum_{\pi} G_{\pi(i_1), \cdots, \pi(i_p)}$, where $\{G_{i_1, \cdots, i_p}\}$ is an array of iid standard Gaussian entries.   This implies that 
\begin{align}
\frac{\mathcal{S}(n,p,d)}{n} = \frac{d}{2p}+\frac{1}{2} \sqrt{\frac{d}{p}}\, \frac{1}{n} \E\Big[ \max_{\us \in \mathcal{X}^n} \sum_{i_1,\cdots, i_p} \frac{G_{i_1,\cdots, i_p}}{n^{(p-1)/2}} \sigma_{i_1}\cdots \sigma_{i_p} \Big] + o(1). \nonumber
\end{align}
The desired result follows immediately from the definition of $\Par_p$ introduced above.

\textit{Proof of Lemma \ref{lemma:Maxqcut}:} 
We will use the condition $(C2)$ of Proposition \ref{prop:equivalence}. To this end, we note that in this case $A_n  = [q]^n$, and setting $f(x,y) = \bone(x \neq y)$, we have,
$V_n/2= \MaxCut(G_n, q)/n $. 
In this case, $\Psi(\bm) = \sum_i m_i (1- m_i)$. It is easy to see that this function attains its unique maximum $(1- 1/q)$ at $\bm^* = \bone/q$. Finally, we have, 
\begin{align}
\bar{\Psi}(\bm) = \sum_{i=1}^{q-1} m_i (1- m_i ) + \Big(\sum_{i=1}^{q-1} m_i \Big) \Big( 1 - \sum_{i=1}^{q-1} m_i \Big) 
= 2 \sum_{i=1}^{q-1} m_i - \sum_{i=1}^{q-1} m_i^2 - \Big( \sum_{i=1}^{q-1} m_i \Big)^2. \nonumber
\end{align}
This immediately implies $- \grad^2 \bar{\Psi}(\bm) = I + \bone \bone^{{\sf T}} \succ I$. This verifies condition $(C2)$. For $\mathbf{p} \in \simp$, we set,  
\begin{align}
\Sigma(\mathbf{p}) = \{ \us \in [q]^n : \sum_{i=1}^{n} \bone(\sigma_i = k ) = n p_k , 1\leq k \leq q \}. \nonumber
\end{align}
Therefore, Proposition \ref{prop:equivalence} implies that 
\begin{align}
\frac{\MaxCut(G_n, q)}{n} &= \frac{d}{2} \Big(1 - \frac{1}{q} \Big) + \frac{\sqrt{d}}{2n} \E\Big[ \max_{\sigma \in \Sigma(\bone/q)} \sum_{i,j} \frac{J_{ij}}{\sqrt{n}} \bone(\sigma_i \neq \sigma_j) \Big] + o_d(\sqrt{d}). \nonumber\\
&= \frac{d}{2} \Big(1 - \frac{1}{q} \Big) + \frac{\sqrt{d}}{2n} \E\Big[ \max_{\sigma \in \Sigma(\bone/q)} \sum_{i,j} \frac{J_{ij}}{\sqrt{n}} \bone(\sigma_i = \sigma_j) \Big] + o_d(\sqrt{d}), \label{eq:maxqcut-intermediate}
\end{align}
where the last equation uses $[\sum_{i,j} J_{ij}/ \sqrt{n} ] /n \pto 0$ as $n \to \infty$ and $(J_{ij})_{\{1\leq i,j \leq n \}} =^d (- J_{ij})_{\{ 1 \leq i, j \leq n \} }$, with $=^d$ denoting equality in distribution. Finally, an application of \cite{panchenko2015potts} implies that as $n \to \infty$, 
\begin{align}
\frac{1}{n} \E\Big[ \max_{\sigma \in \Sigma(\bone/q)} \sum_{i,j} \frac{J_{ij}}{\sqrt{n}} \bone(\sigma_i = \sigma_j) \Big] \to \Par_*(q). \nonumber
\end{align}

Plugging this back into \eqref{eq:maxqcut-intermediate} immediately concludes the proof.

\textit{Proof of Lemma \ref{lemma:sbm}:}
This example is not exactly in the framework of the general problem introduced in Theorem \ref{lemma:comparison}. However, we will establish that the same techniques are invaluable in this case. 
Without loss of generality, we assume that $S=\{1,2, \cdots, n/2 \}$. We can encode each partition by an assignment of spins $\us = (\sigma_1 , \cdots, \sigma_n) \in \{\pm 1\}^n$. The constraint that the two halves must be of equal size enforces that $\sum_i \sigma_i=0$. In this case, we have, denoting the adjacency matrix of $G(n, a/n, b/n)$ by $A$ and setting $A^{\cen} = A - \E[A]$,
\begin{align}
\frac{\mcut(G_n)}{n} &= \frac{1}{2n} \min_{\{\us: \sigma_i=0 \}} \sum_{i,j} A_{ij} \bone(\sigma_i \neq \sigma_j) \nonumber\\
&= \frac{1}{2n} \min_{\{\sigma: \sum \sigma_i =0 \}} \sum_{i,j}\Big\{ \frac{d}{n} \bone(\sigma_i \neq \sigma_j )+ \frac{\xi}{n} \sqrt{d} \chi_{ij} \bone(\sigma_i \neq \sigma_j) + A_{ij}^{\cen} \bone(\sigma_i \neq \sigma_j) \Big\}, \label{eq:sbm_intermediate1}
\end{align}
where $\chi_{ij}= 1$ if both $i,j \in S$ or both $i,j \in S^c$ and $\chi_{ij} = -1$ otherwise. We note that for each $\us \in \{-1, +1\}^n$ satisfying $\sum_i \sigma_i =0$, $\sum_{i,j} \bone(\sigma_i \neq \sigma_j) = n^2/2$. Therefore, we have, from \eqref{eq:sbm_intermediate1}
\begin{align}
\frac{\mcut(G_n)}{n} = \frac{d}{4} + \frac{\sqrt{d}}{2n} \min_{\{\sigma: \sum \sigma_i =0 \}}  \sum_{i,j}  \Big\{ \frac{\xi}{n} \chi_{ij} + \frac{A_{ij}^{\cen}}{\sqrt{d}} \Big\} \bone(\sigma_i \neq \sigma_j). \label{eq:sbm_intermediate2}
\end{align}
We note that for boolean variables, we have, $\bone(\sigma_i \neq \sigma_j) = ( 1- \sigma_i \sigma_j )/2$. Further, we observe that $\sum_{i,j}  \chi_{ij} =0$ and $\Var(\sum_{i,j} A_{ij}^\cen)=O(n)$ which implies that $\frac{1}{n}\sum_{i,j} A_{ij}^{\cen} \pto 0$ as $n\to \infty$. We set $v_i = 1$ if $i \in S$ and $v_i = -1$ otherwise, such that $\chi_{ij} =v_i v_j$. Thus we have, from \eqref{eq:sbm_intermediate2},
\begin{align}
\frac{\mcut(G_n)}{n} = \frac{d}{4} - \frac{\sqrt{d}}{4n} \max_{\{\us: \sum \sigma_i =0 \} } \Big[ \frac{\xi}{n} \langle v, \us \rangle^2 + \frac{1}{\sqrt{d}} \sum_{i,j} A_{ij}^{\cen} \sigma_i \sigma_j \Big]. 
\end{align}
At this point, we employ the following comparison principle, which is the analogue of Theorem \ref{lemma:comparison} in this context. The proof is similar to Theorem \ref{lemma:comparison} and thus we will simply sketch the proof later in the section. It might be useful to study the proof of Theorem \ref{lemma:comparison} before reading the proofs below.  

\begin{thm}
\label{thm:comparison-sbm}
With high probability as $n\to \infty$, we have, 
\begin{align}
\frac{\mcut(G_n)}{n} = \frac{d}{4} - \frac{\sqrt{d}}{4n} \E\Big[  \max_{\{\us: \sum \sigma_i =0 \} }  \Big\{ \frac{\xi}{n} \langle v, \us \rangle^2 + \sum_{i,j} \frac{J_{ij}}{\sqrt{n}} \sigma_i \sigma_j\Big\} \Big] + o_d(\sqrt{d}) , \nonumber
\end{align}
where $(J_{ij})$ is a standard {{\sf GOE}} matrix. 
\end{thm}


Applying Theorem \ref{thm:comparison-sbm}, we have, 
\begin{align}
\frac{\mcut(G_n)}{n} &= \frac{d}{4} - \frac{\sqrt{d}}{4n} \E\Big[  \max_{\{\us: \sum \sigma_i =0 \} }  \Big\{ \frac{\xi}{n} \langle v, \us \rangle^2 + \sum_{i,j} \frac{J_{ij}}{\sqrt{n}} \sigma_i \sigma_j\Big\} \Big] + o_d(\sqrt{d}) \nonumber\\
&= \frac{d}{4} - \frac{\sqrt{d}}{4n} \E\Big[  \max_{\us \in \mathcal{C}_n  }  \Big\{ \frac{\xi}{n} \langle \bone, \us \rangle^2 + \sum_{i,j} \frac{J_{ij}}{\sqrt{n}} \sigma_i \sigma_j\Big\} \Big] + o_d(\sqrt{d}) \nonumber,
\end{align}
where $\mathcal{C}_n = \{ \us: \sum_{i=1}^{n/2} \sigma_i = \sum_{i=n/2+1}^{n} \sigma_i \}$. 
Finally, the proof can be completed by an application of the following lemma. 
\begin{lemma} 
\label{lemma:curieweiss-sk-comparison}
We have, as $n \to \infty$, 
\begin{align}
\frac{1}{n} \E \Big[ \max_{\us \in \mathcal{C}_n  }  \Big\{ \frac{\xi}{n} \langle \bone, \us \rangle^2 + \sum_{i,j} \frac{J_{ij}}{\sqrt{n}} \sigma_i \sigma_j\Big\}  \Big] = \frac{1}{n} \E \Big[ \max_{\us \in \{\pm 1\}^n }  \Big\{ \frac{\xi}{n} \langle \bone, \us \rangle^2 + \sum_{i,j} \frac{J_{ij}}{\sqrt{n}} \sigma_i \sigma_j\Big\}  \Big] + o(1). 
\end{align}
\end{lemma}

It remains to establish Theorem \ref{thm:comparison-sbm} and Lemma \ref{lemma:curieweiss-sk-comparison}. We first outline the proof of Theorem \ref{thm:comparison-sbm} and defer the proof of Lemma \ref{lemma:curieweiss-sk-comparison} to the end of the section.


\textit{Proof of Theorem \ref{thm:comparison-sbm}:}
Given any symmetric matrix $M$ and for any configuration $\us \in \{ \pm 1 \}^n $ satisfying $\sum_i \sigma_i =0$, we define, 
\begin{align}
H(\us, M) = \sum_{i,j} M_{ij} \sigma_i \sigma_j , \,\,\,\,
\Phi (\beta, M) = \log \Big[ \sum_{\us: \sum \sigma_i =0} \exp(\beta H(\us, M) ) \Big]. \nonumber 
\end{align}
We define the symmetric Gaussian matrix $B= \frac{\xi}{n} v v^{{\sf T}} + \frac{J}{\sqrt{n}}$, where $J = (J_{i,j})$ is a standard {{\sf GOE }} matrix. We will establish 
\begin{align}
\Big| \frac{1}{n\beta} \E \Big[ \Phi(\beta, \frac{\xi}{n} v v^{{\sf T}} + \frac{A_G^{\cen}}{\sqrt{d}} ) \Big] - \frac{1}{n \beta} \E\Big[\Phi (\beta, B) \Big] \Big| \leq \frac{C\beta^2}{\sqrt{d}}  \label{eq:comparison-sbm}
\end{align}
for some constant $C>0$. The thesis follows subsequently by using Lemma \ref{lemma:finitetemp_approx} with $|\mathcal{X}|=2$. 
To this end, we proceed in two steps, and define an intermediate Gaussian random matrix
\begin{align}
\mathbf{D}(\xi) = \frac{\xi}{n} vv^{{\sf T}} + \mathbf{U}\, ,
\end{align}
where $\mathbf{U} = \mathbf{U}^{{\sf T}}\in\reals^{n\times n}$ is a Gaussian random
matrix with $\{U_{ij}\}_{1\le i\le j\le n}$ independent zero-mean Gaussian random
variables with
 \begin{align}
\Var(U_{ij}) = \begin{cases}
a[1-a/n]/(nd)& \mbox{ if $\{i,j\}\subseteq S$ or $\{i,j\}\subseteq S^c$,}\\
b[1-b/n]/(n d) & \mbox{ if $i\in S,\, j\in S^c$ or $i\in
  S^c ,j\in S$,}
\end{cases}\label{eq:UnequalVariances}
\end{align}
and $U_{ii}=0$.
By triangular inequality
\begin{align}
\left|\frac{1}{n}\E\Phi\Big(\beta, \frac{\xi}{n} v v^{{\sf T}} + A^{\cen}_G/\sqrt{d}\Big)-\frac{1}{n}\E\Phi\big(\beta,\mathbf{B}\big)\right|&\le 
\left|\frac{1}{n}\E\Phi\Big(\beta, \frac{\xi}{n} v v^{{\sf T}} + A^{\cen}_G/\sqrt{d}\Big)-\frac{1}{n}\E\Phi\big(\beta,\mathbf{D}\big)\right|\nonumber\\
&+
\left|\frac{1}{n}\E\Phi\big(\beta,k;\mathbf{D}\big)-\frac{1}{n}\E\Phi\big(\beta,k;\mathbf{B}\big)\right|
\, .
\end{align}

The proof of \eqref{eq:comparison-sbm} follows therefore from the next two results. 
\begin{lemma}\label{lemma:Interpolation1}
With the above definitions, if $n\ge (15d)^2$, then
\begin{align}
\left|\frac{1}{n}\E\Phi\Big(\beta, \frac{\xi}{n} v v^{{\sf T}} + \frac{A^{\cen}}{\sqrt{d}}\Big)-\frac{1}{n}\E\Phi\big(\beta,\mathbf{D}\big)\right|\le \frac{2\beta^3}{\sqrt{d}}\, .
\end{align}
\end{lemma}

\begin{lemma}
\label{lemma:Interpolation2}
With the above definitions, there exists an absolute constant $n_0$
such that, for all $n\ge n_0$, 
\begin{align}
\label{eq:interpolation2}
\left|\frac{1}{n\beta}\E\Phi\big(\beta,\mathbf{B}\big)-\frac{1}{n \beta}\E\Phi\big(\beta,\mathbf{D}\big)\right|\le 5\sqrt{\frac{a-b}{d}}\, .
\end{align}
\end{lemma}

The proof of Lemma \ref{lemma:Interpolation1} is the same as that of Lemma \ref{lemma:interpolation} and will thus be omitted. Lemma \ref{lemma:Interpolation2} is proved in \cite[Lemma E.2]{montanari2015semidefinite} and will thus be omitted. 
%
%
%
%
%
Finally, we prove Lemma \ref{lemma:curieweiss-sk-comparison}. 

\textit{Proof of Lemma \ref{lemma:curieweiss-sk-comparison}: }
Trivially, we have $\E [ \max_{\us \in \{\pm 1\}^n }  \{ \frac{\xi}{n} \langle \bone, \us \rangle^2 + \sum_{i,j} \frac{J_{ij}}{\sqrt{n}} \sigma_i \sigma_j\}  ]  \geq \E [ \max_{\us \in \mathcal{C}_n  }  \{ \frac{\xi}{n} \langle \bone, \us \rangle^2 + \sum_{i,j} \frac{J_{ij}}{\sqrt{n}} \sigma_i \sigma_j\}  ]$. To derive the opposite bound, we proceed as follows. 
Let 
\begin{align}
\us^* = \argmax_{\{\pm 1\}^n} \Big[  \frac{\xi}{n} \langle \bone, \us \rangle^2 + \sum_{i,j} \frac{J_{ij}}{\sqrt{n}} \sigma_i \sigma_j \Big]  =  \argmax_{\{ \pm 1 \}^n } \sum_{i,j} M_{ij} \sigma_i \sigma_j . \nonumber 
\end{align}
where $M = (M_{i,j})$ is a symmetric matrix, $\{M_{i,j} : i<j\}$ are independent $\dN(\xi/n , 1/n)$ random variables. The definition of $\us^*$ implies that $\sigma_i^* = \sign ( \sum_j M_{ij} \sigma_j^*)$. Thus setting $f_i = \sum_j M_{ij}\sigma_j^*$, we have,
\begin{align}
\sum_{i,j} M_{i,j} \sigma_i ^* \sigma_j^* = \sum_i |f_i |. \nonumber 
\end{align}
Now, elementary bounds on the spectral norm of a one-rank perturbed random matrix \cite{feral2007largest} implies that with probability $1$, $\sum_i |f_i | \leq C(\xi) n$ for some universal constant $C(\xi)$ independent of $n$. Finally this implies that with probability $1$, the set $R^*= \{ i \in [n]: |f_i| \leq 10 C(\xi) \}$ has size at least $9n/10  $.
We define $m^* = \frac{1}{n} \sum_i \bone(\sigma_i^*=1)$. By the symmetry of the problem, given $m^*$, $\us^*$ is uniformly distributed on $\{ \us: \sum_i \bone(\sigma_i) = n m^*\}$. Thus $\sum_{i=1}^{n/2} \bone( \sigma_i=1) = H $, where $H \sim {{\sf Hypergeometric}}(n, m^*,n/2)$. It is easy to see that $\Var[H  | m^*] \lesssim n$ and therefore, by Chebychev inequality, with high probability, $|\sum_{i =1}^{n/2} \bone(\sigma_i =1) - \sum_{i= n/2 +1}^{n} \bone(\sigma_i=1) | \lesssim \sqrt{n\log n}$. Thus with high probability, we can flip at most $O(\sqrt{n \log n})$ bits of $\us^*$ to get a configuration in $\mathcal{C}_n$. We will necessarily flip these coordinates from $R^*$ and denote the set of flipped indices by $W$.   Let the derived configuration be 
$\us^{\star}$. Then we have
\begin{align}
|\sum_{i,j} M_{ij} \sigma_i^* \sigma_j^* - \sum_{i,j} M_{i,j} \sigma_i^{\star} \sigma_j^{\star} | =2\, | \sum_{i \in W} \sum_{j \in W^c} M_{i,j} \sigma_i^* \sigma_j^*|  
\leq 20 C(\xi) |W| + 2\, \sum_{i,j \in W} |M_{i,j}|. \nonumber 
\end{align}
We have, $\sum_{i,j \in W} | M_{i,j}| \lesssim \xi \log n + \sum_{i,j \in W} |J_{ij}| / \sqrt{n}$, where $\{J_{ij}: i < j\}$ are independent standard Gaussian random variables. With high probability, $|W| \leq  C\sqrt{n \log n}$ for some constant $C>0$ arbitrarily large. We will show that $\max_{S \subset [n], |S| \leq C\sqrt{n \log n} } \sum_{i < j \in S} |J_{ij}|/{\sqrt{n}} = o(n)$ with high probability. To this end, we note that for a fixed $S$, we have, by Markov's inequality, 
\begin{align}
\P\Big[ \sum_{i < j \in S} \frac{|J_{ij}|}{\sqrt{n}} \geq \delta_n \Big] \leq  \exp(- \sqrt{n}\delta_n) \E[e^{|J|}]^{C^2 n \log n} \leq \exp(- c \sqrt{n}\delta_n) 
\end{align}
whenever $\delta_n \gg n^{1/2 + \delta}$ for any $\delta>0$. The desired claim now follows by a union bound over at most $2^n$ possible $S$.  
Finally we have,
\begin{align}
\E \Big[ \max_{\us \in \mathcal{C}_n  }  \Big\{ \frac{\xi}{n} \langle \bone, \us \rangle^2 + \sum_{i,j} \frac{J_{ij}}{\sqrt{n}} \sigma_i \sigma_j \Big\}  \Big] \geq \E \Big[ \sum_{i,j} M_{ij} \sigma_i^{\star} \sigma_j^{\star} \Big] \geq \E \Big[ \max_{\us \in \{\pm 1\}^n }  \Big\{ \frac{\xi}{n} \langle \bone, \us \rangle^2 + \sum_{i,j} \frac{J_{ij}}{\sqrt{n}} \sigma_i \sigma_j \Big\}  \Big] - o(n) \nonumber
\end{align} 
thereby completing the proof.

\section{Proof of Theorem \ref{lemma:comparison}}
\label{section:main-proof}

We prove Theorem \ref{lemma:comparison} in this section. We mainly use the Lindeberg interpolation strategy, which has been widely used to prove universality in probability. We define 
\begin{align}
H_1(\us) &= \frac{1}{\sqrt{d}}\sum_{i_1 \neq i_2 \neq \cdots \neq i_p}  A_{i_1, \cdots, i_p}  f(\sigma_{i_1}, \cdots , \sigma_{ i_p}). \nonumber\\ 
H_2(\us) &=  \sum_{i_1\neq i_2 \neq \cdots \neq i_p} \Big[\sqrt{d}\frac{\kappa_1(i_1,\cdots, i_p)}{n^{p-1}}+ \frac{J_{i_1, \cdots, i_p}}{n^{(p-1)/2}}\Big] f(\sigma_{i_1}, \cdots, \sigma_{i_p}) , \nonumber 
\end{align}
where $J_{i_1, \cdots, i_p} \sim \mathcal{N}(0, \kappa_2(i_1, \cdots, i_p))$ are independent random variables for $i_1< i_2 < \cdots < i_p$ and for any permutation $\pi$, $J_{i_1, \cdots, i_p}= J_{\pi(i_1), \cdots, \pi(i_p)}$. 
We note that $V_n = \frac{\sqrt{d}}{n} \max_{\us \in A_n}  H_1(\us)$. Our first lemma establishes that $V_n$ is concentrated tightly around its expectation. 

\begin{lemma}
\label{lemma:concentration}
We have, as $n\to \infty$, 
$V_n - \E[V_n] \pto 0$ . 
\end{lemma}

\begin{proof}
To control the variance of $V_n$, we use the Efron-Stein inequality. We note that if we replace $A_{i_1, \cdots, i_p}$ by an independent copy $A_{i_1, \cdots, i_p}'$, 
\begin{align}
\E[(A_{i_1,\cdots, i_p} - A_{i_1,\cdots, i_p}')^2] \leq 2\E[A_{i_1,\cdots, i_p}^2] = \frac{2\kappa_2(i_1, \cdots, i_p)}{n^{p-1}}. \nonumber 
\end{align}
This implies, by Efron-Stein inequality \cite{blm}  
\begin{align}
\Var[\max_{\us \in A_n} H_1 (\us) ] \leq \frac{\| f \|_{\infty}}{d} \sum_{1\leq i_1 \neq \cdots \neq i_p \leq n} \frac{\kappa_2(i_1, \cdots, i_p)}{n^{p-1}} = O(n). \nonumber
\end{align}
This immediately implies that $\Var(V_n) = O(1/n)$. 
\end{proof}

Thus it suffices to work with the expected values. We define
\begin{align}
e_{1,n} = \frac{1}{n} \E \Big[\max_{\us \in A_n} H_1(\us)\Big]  , \,\,\,\,\,  e_{2,n} = \frac{1}{n} \E \Big[\max_{\us \in A_n} H_2(\us)\Big] \nonumber
\end{align}

We introduce the following smooth approximation of the maximum values. 

\begin{align}
\Phi_1(\beta) = \frac{1}{n} \E\Big[\log \sum_{\us \in A_n} \exp( \beta H_1(\us)) \Big],  \,\,\,\,
\Phi_2(\beta) = \frac{1}{n} \E \Big[\log \sum_{\us \in A_n} \exp( \beta H_2(\us)) \Big]. \label{eq:free_energy}
\end{align}

We can derive the following bound on the difference of $\Phi_1$ and $\Phi_2$. 
\begin{lemma}
\label{lemma:interpolation}
There exists a constant $D>0$ independent of $n$ such that 
\begin{align}
\frac{1}{\beta} |\Phi_1(\beta) - \Phi_2(\beta) | \leq \frac{D\beta^2}{\sqrt{d}}. 
\end{align}
\end{lemma}

We note that $\Phi_i(\beta)/ \beta \to e_{i,n}$ for $i =1,2$ as $\beta \to \infty$. The following lemma gives us a quantitative version of this statement, valid uniformly for all $n$. 

\begin{lemma}
\label{lemma:finitetemp_approx}
We have, for $i=1,2$, for all $n$ sufficiently large, 
\begin{align}
\Big| \frac{\Phi_i(\beta)}{\beta} - e_{i,n} \Big| \leq  \frac{ \log |\mathcal{X}|}{\beta}.  \nonumber
\end{align}
\end{lemma}

The proofs of Lemmas \ref{lemma:interpolation} and \ref{lemma:finitetemp_approx} will be deferred to the end of this section. We complete the proof of Lemma \ref{lemma:comparison} using these results. 
To this end, we note that using Lemma \ref{lemma:finitetemp_approx}, we have,
\begin{align}
|e_{1,n} - e_{2,n}| \leq \frac{D\beta^2}{\sqrt{d}} + 2\, \frac{ \log |\mathcal{X}| }{\beta}.  \nonumber 
\end{align}
Thus choosing $\beta = d^{1/4- \delta}$ for some $0< \delta < 1/4$, we have $|e_{1,n} - e_{2,n}| = o_d(\sqrt{d})$. Now, we have, from \eqref{eq:opt}
\begin{align}
V_n &= \frac{1}{n} \max_{\us \in A_n} \sum_{i_1, \cdots, i_p} A_{i_1, \cdots, i_p} f(\sigma_{i_1}, \cdots, \sigma_{i_p}) \nonumber \\
&= \E[\max_{\us \in A_n(\alpha)} [ \alpha d +T_n^{\alpha} \sqrt{d} ] ]+ o_d(\sqrt{d}), \nonumber 
\end{align}
where $A_n(\alpha) = \{ \us \in A_n : \frac{1}{n^p}\sum_{i_1 \neq \cdots \neq i_p} \kappa(i_1, \cdots, i_p) f(\sigma_{i_1},\cdots, \sigma_{i_p}) = \alpha \} $. This completes the proof of the lemma.

\subsection{Proof of Lemma \ref{lemma:interpolation}} 
We will use the following version of the Lindeberg invariance principle \cite{chatterjee2005simple}. 

\begin{lemma}
\label{lemma:lindeberg}
Let $F: \mathbb{R}^N \to \mathbb{R}$ be three times continuously differentiable. Let $\mathbf{X} = (X_1, \cdots , X_N)$ and $\mathbf{Z} = (Z_1, \cdots , Z_N)$ be two vectors of independent random variables satisfying $\E[X_i] = \E[Z_i]$ and $\E[X_i^2] = \E[Z_i^2]$ for all $1\leq i \leq N$. Then we have,
\begin{align}
| \E[F(X)] - \E[F(Z)] | \leq \frac{1}{6} S_3 \max_{ 1 \leq i \leq N} \| \partial_i^3 F \|_{\infty}, \nonumber 
\end{align}
where $S_3 = \sum_{i=1}^{N} [\E|X_i|^3 + \E|Z_i |^3]$ and $ \| \partial_i^3 F \|_{\infty}  = \sup | \frac{\partial^3}{\partial x_i^3} F(x) |$. 
\end{lemma}

Given $\mathbf{M} = \{ M_{i_1, \cdots, i_p } : 1\leq i_1 \neq i_2 \neq \cdots \neq i_p \leq n \}$, define 
\begin{align}
H(\us, \mathbf{M}) &= \sum_{i_1 \neq \cdots \neq i_p} M_{i_1, \cdots, i_p} f(\sigma_{i_1} , \cdots, \sigma_{i_p})  \nonumber\\ 
G(\mathbf{M} ) &= \frac{1}{n\beta} \E\Big[ \log \sum_{\us \in A_n } \exp( \beta H(\us, \mathbf{M})) \Big] . \nonumber 
\end{align}

Therefore, setting $\mathbf{A} = \{\frac{1}{\sqrt{d}} A_{i_1, \cdots, i_p} : 1\leq i_1 < \cdots < i_p \leq n \}$ and $\mathbf{J} = \{ \sqrt{d}\frac{\kappa_1(i_1 ,\cdots i_p)}{n^{p-1}}+ \frac{J_{i_1, \cdots, i_p}}{n^{(p-1)/2}} : 1\leq i_1 < \cdots < i_p \leq n \}$, we have, by a slight abuse of notation, 
\begin{align}
\frac{\Phi_1(\beta)}{\beta} = \E[ G(\mathbf{A})], \,\,\,\,\,\, \frac{\Phi_2(\beta)}{\beta} = \E[ G( \mathbf{J}) ] \nonumber. 
\end{align}

An application of Lemma \ref{lemma:lindeberg} yields
\begin{align}
\frac{1}{\beta} | \Phi_1(\beta) - \Phi_2(\beta) | \leq \frac{1}{6} S_3 \max_{i_1, \cdots, i_p} \| \partial_{i_1, \cdots, i_p}^3 G \|_{\infty}. \nonumber 
\end{align}
Let $\< \cdot \>$ denote the expectation with respect to the Gibbs measure $\mu_{\mathbf{M}}(\us) \propto \exp(\beta H(\us, M))$ under the weight sequence $\mathbf{M}$. Direct computation yields, for $i_1 < i_2 < \cdots < i_p$, 
\begin{align}
\partial_{i_1, \cdots, i_p} G &= \frac{p!}{n}\, \< f \>,  \nonumber \\
\partial_{i_1, \cdots, i_p }^2 G &= \frac{\beta (p!)^2}{n} \, [ \< f^2 \> - \< f \>^2], \nonumber \\
\partial_{i_1, \cdots, i_p}^3 G &= \frac{\beta^2 (p!)^3}{n} \,[ \< f^3 \> - 3 \< f ^2 \> \<f \> + 2 \< f \>^3].  \nonumber 
\end{align}
Thus we have, $\max_{i_1, \cdots, i_p} \| \partial_{i_1, \cdots, i_p}^3 G \|_{\infty} \leq \frac{6\beta^2 (p!)^3}{n} \|f \|_{\infty}^3 $. 
Finally, we have,
\begin{align}
S_3 &= \sum_{i_1 <\cdots < i_p} \Big[\frac{1}{d^{3/2}} \E| A_{i_1, \cdots, i_p} |^3 + \E|\sqrt{d}\frac{\kappa_1(i_1,\cdots, i_p)}{n^{p-1}}+ \frac{J_{i_1, \cdots, i_p}}{n^{(p-1)/2}}|^3 \Big] = I + II.  \nonumber 
\end{align}
We bound each term separately. To bound the first term, we note,
\begin{align}
I \leq \frac{1}{d^{3/2}} \sum_{i_1< \cdots < i_p} \E|A_{i_1, \cdots, i_p}|^3 
\leq \frac{ B_U}{d^{3/2}} \sum_{i_1 < \cdots < i_p} d\, \frac{\kappa_2(i_1,\cdots, i_p)}{n^{p-1}}  
\lesssim \frac{n}{\sqrt{d}} . \nonumber 
\end{align}
Finally, to bound the second term, we note that, 
\begin{align}
II \lesssim \frac{n^p}{n^{3(p-1)}} + \frac{n^p}{n^{3(p-1)/2}} \lesssim \frac{n^{3/2}}{n^{p/2}} + o(1)= o_n(n) . \nonumber 
\end{align}

This completes the proof.

\subsection{Proof of Lemma \ref{lemma:finitetemp_approx}}
Let $H: \mathcal{X}^n \to \mathbb{R}$ be any function and for any subset of configurations $A_n$,  define the ``partition function" $Z_n(\beta) =\sum_{\us \in A_n} \exp(\beta H(\us))$. Further, define the Gibbs measure
$\mu_{\beta,n}(\us) = \exp(\beta H(\us))/ Z_n(\beta)$ and the log-partition function $\phi_n(\beta) = \frac{1}{n} \log Z_n(\beta)$. Now, we observe that 
\begin{align}
\frac{\partial}{\partial \beta} \frac{\phi_n(\beta)}{\beta} = -\frac{1}{n\beta^2}S(\mu_{\beta,n}), \nonumber
\end{align}
where $S(\mu_{\beta,n})= - \sum_{\us \in A_n} \mu_{\beta,n}(\us) \log \mu_{\beta,n}(\us)$ is the entropy of the distribution $\mu_{\beta,n}$. Now, we have, $S(\mu_{\beta,n}) \leq \log | A_n |$, where $| \cdot |$ denotes the cardinality of the configuration space. Finally, noting that $|A_n| \leq |\mathcal{X}|^{n}$, we immediately have, $\frac{\partial}{\partial \beta} \frac{\phi_n(\beta) }{\beta} \in [-\frac{\log |\mathcal{X}|}{\beta^2}, 0] $.  This immediately implies
\begin{align}
|e_{i,n} (\alpha)- \frac{\Phi_i (\beta)}{\beta}| &= | \int_{\beta}^{\infty} \frac{\partial}{\partial t}\Big(\frac{\Phi_i(t)}{t} \Big)  \de t | \leq \frac{\log |\mathcal{X}| }{\beta}. \nonumber 
\end{align}

This completes the proof.

\section{Proof of Theorem \ref{thm:regular}}
\label{section:proof-regular}

First, we study a modified optimization problem on \ER hypergraphs, which will be crucially related to the behavior of the original problem \eqref{eq:opt} on regular instances. To this end, consider the $p$-uniform \ER hypergraph constructed as follows. The hypergraph has vertex set $V= \{1, \cdots, n\}$, and each $p$-subset of $V$ is added independently to the set of hyperedges with probability $\frac{(p-1)!(d- C\sqrt{d}\log d )}{ n^{p-1}}$. We denote the adjacency matrix of this hypergraph as $A_{H}$. 
Given any function $f: \mathcal{X}^p \to \reals$ which is symmetric in its arguments, consider the optimization problem 
\begin{align}
V_n(A_H) =  \frac{1}{n} \max_{\us \in A_n} \Big[ &\frac{d}{n^{p-1}}\sum_{i_1,i_2, \cdots, i_p} f(\sigma_{i_1},\cdots, \sigma_{i_p}) + \sum_{i_1,\cdots, i_p} A_{H}^{\cen} (i_1,\cdots, i_p) \bar{f(}\sigma_{i_1},\cdots, \sigma_{i_p}) \Big], \label{eq:d-reg_surrogate}\\
&\bar{f}(\sigma_{i_1}, \cdots, \sigma_{i_p}) = f(\sigma_{i_1}, \cdots, \sigma_{i_p}) - \frac{p}{n^{p-1}} \sum_{l_2, \cdots, l_p} f(\sigma_{i_1}, \sigma_{l_2} ,\cdots, \sigma_{l_p}), \nonumber 
\end{align}
where we set $A_{H}^{\cen}= A_H - \E[A_H]$. 
The next lemma lemma approximates the value of this optimization problem using an appropriate Gaussian surrogate. 
\begin{lemma} 
\label{lemma:contiguity}
For $d$ sufficiently large, as $n\to \infty$, with high probability,  
\begin{align}
V_n(A_H) = \E \max_{\alpha}[ d \alpha + S_n^{\alpha} \sqrt{d}] + o_d(\sqrt{d}), 
\end{align}
with $S_n^{\alpha}$ as defined in \eqref{eq:regular_opt}. 
\end{lemma}

\noindent
The proof is similar to Theorem \ref{lemma:comparison}, and is therefore omitted.

Next, we initiate the study of \eqref{eq:opt} on $p$- uniform, $d$-regular instances, and first observe that we can equivalently work with the hypergraph version of the configuration model \cite{bollobas_configuration}. The probability that the hypergraph is simple is lower bounded by 
$O(1)$ \cite{cooper_frieze_molloy_reed}. Thus, for our purposes, it suffices to establish the result for the configuration model.

We recall the usual construction of the random $d$ regular, $p$ uniform hypergraph under the configuration model. 
A multi-hypergraph $G = (V, E)$, where $V$ is the set of vertices, and $E$ is the set of hyperedges. In a $p$- uniform multi-hypergraph, each $e \in E$ is a subset of size $p$ from $V$, with possible repetitions. 
 We will assume throughout that $p | nd$. Under the configuration model,
 we consider $nd$ objects, labeled as 
\begin{align}
\mathscr{L} = \{ (i, j ): 1\leq i \leq n, 1 \leq j \leq d \}. \nonumber 
\end{align}
We refer to the objects $\{ (i,j ) : 1 \leq i \leq d\}$ as the ``clones" of vertex $i$. Consider also the set of half-edges
\begin{align}
\mathscr{C} = \{ (a, k ) : 1\leq a \leq \frac{nd}{p}, 1 \leq k \leq p \}. \nonumber 
\end{align}
By a $d$-regular, $p$- uniform hypergraph drawn from the configuration model, we refer to a uniform random matching $\upsilon :\mathscr{C} \to \mathscr{L}$ (formally, the matching is a random bijection between the sets $\mathscr{C}$ and $\mathscr{L}$). In this case, we set $V= \{1, 2, \cdots, n\}$ and define  $E= \{ \{ \upsilon((a,k)) : 1 \leq k \leq p \}, 1 \leq a \leq \frac{nd}{p}  \}$. We refer to the hyperedge $\{\upsilon((a,k)): 1 \leq k \leq p \}$ as the hyperedge $a$. 

The main challenge in the analysis of the uniform $d$-regular hypergraph stems from the dependence in the hyperedges. Our main idea, similar to the one introduced in \cite{dembo2015extremal}, is to relate the optimal value \eqref{eq:opt} on the regular hypergraph to \eqref{eq:d-reg_surrogate}, up to $o_d(\sqrt{d})$ corrections. Theorem \ref{thm:regular} then follows directly from Lemma \ref{lemma:contiguity}. 

\textit{Proof of Theorem \ref{thm:regular}:}
The main idea is to ``find" an \ER hypergraph, with slightly smaller average degree, ``embedded" in the uniform $d$ regular $p$-uniform hypergraph, and relate the optimization problem on the larger graph to a modified problem on the smaller embedded graph. We formalize this idea in the rest of the proof. During the proof, we will sometimes construct hypergraphs which are not $p$ uniform, in that they have an hyperedge with less than $p$ elements. We note that this slight modification does not affect our conclusions in any way. 

We will crucially use the following two stage construction of the configuration model. Throughout, we define the vertex set $V= \{1, \cdots, n\}$. Let $C>0$ be a large constant, to be chosen later. Let $X_i \sim \dPois(d- C\sqrt{d} \log d)$ i.i.d. for some $C>0$ sufficiently large and we set $Z_i = (d- X_i)_{+}$. Recall the clones $\{ (i, j ): 1\leq j \leq d \}$ used in the construction of the configuration model. For $1 \leq i \leq n$, we color the clones $\{(i, j): 1 \leq j \leq Z_i \}$ by the color \BLUE, and the rest are colored \RED. The multi-hypergraph $G_1$ is formed by a uniform random matching $\upsilon$ between $\mathscr{L}$ and $\mathscr{C}$, and is thus distributed as a configuration model. Consider the set of hyperedges $a$ in $G_1$ such that $\{ \upsilon(a,k) : 1 \leq k \leq p \}$ are \RED\, clones and denote the sub hypergraph induced by these hyperedges as $\Gr = (V, E(\Gr)) $. Similarly, the sub hypergraph induced by the hyperedges $a$ in $G_1$ such that $\{\upsilon(a,k) : 1 \leq k \leq p \}$ has at least one \BLUE\, clone will be denoted by $\Gb$. We will now delete all the $\{(i,j): 1 \leq i \leq n, 1 \leq j \leq d \}$ clones colored \BLUE, and delete the half-edges  $\{a: 1 \leq a \leq \frac{nd}{p} \}$  in $G_1$ such that $\{\upsilon((a,k)): 1 \leq k \leq p \}$ has at least one \BLUE\, clone. Assume that this operation creates $u = lp + r$, $ 0 \leq l \leq \frac{nd}{p}$, $ r < p$ unmatched \RED\, clones in $G_1$. We add new half-edges $\{ (a, k) : \frac{nd}{p} +1  \leq a \leq \frac{nd}{p}+ l, 1 \leq k \leq p \} \cup \{( \frac{nd}{p} + l+1, k) : 1 \leq k \leq r\}$ and match the new half-edges uniformly to the unmatched \RED\, clones in $G_1$. We refer to the sub-graph induced by the new hyperedges $\{ a : \frac{nd}{p} +1 \leq a \leq \frac{nd}{p} + l+1  \}$ as $\tGr = (V, E(\tGr))$ and define $G_2$ as the multi hypergraph with vertex set $V$ and hyperedges $E(\Gr) \cup E(\tGr)$. We establish in Lemma \ref{lemma:matching} that $G_2$ is equivalently obtained using a random matching between the \RED\, clones $\{(i,j) : 1 \leq i \leq n, Z_i +1  \leq j \leq d \}$ clones and $p$-uniform hyperedges with the same number of half-edges.  Let $A_{G_1}$ and $A_{G_2}$ denote the adjacency matrices of the multi-hypergraphs respectively--- thus for $j=1,2$, $(p-1)! \,A_{G_j}(i_1,\cdots,i_p)$ counts the number of $\{i_1, \cdots, i_p\}$ hyperedges present in $G_j$. 
 
%

  The cornerstone of the proof is the following lemma. We defer the proof for ease of exposition. 

\begin{lemma}
\label{lemma:regular_intermediate}
We have, with high probability as $n\to \infty$, for $d$ sufficiently large, 
\begin{align}
V_n^{R} &=  \frac{1}{n} \max_{\us \in A_n} \Big[ \frac{d}{n^{p-1}}\sum_{i_1,i_2, \cdots, i_p} f(\sigma_{i_1},\cdots, \sigma_{i_p}) + \sum_{i_1,\cdots, i_p} A_{G_2}^{\cen} (i_1,\cdots, i_p) \bar{f}(\sigma_{i_1},\cdots, \sigma_{i_p}) \Big] + o_d(\sqrt{d}), \label{eq:regular_exp_new}
\end{align}
\end{lemma}
where we set $A_{G_2}^{\cen} = A_{G_2} - \E[A_{G_2}]$, and $\bar{f}$ is the same as in \eqref{eq:d-reg_surrogate}. 

%
%
%
%

Given Lemma \ref{lemma:regular_intermediate}, we complete the proof as follows. We will establish that up to $o_d(\sqrt{d})$ corrections, the value of the {\tiny\sf{RHS}} of \eqref{eq:regular_exp_new} is equal to that on an \ER hypergraph with average degree $d- C\sqrt{d}\log d$. This step is accomplished by direct graph comparison arguments. To this end, note that for any two $p$-hypergraphs $G_1 = (V, E_1), G_2 = (V, E_2)$, denoting the optimal value  \eqref{eq:opt} as $V_n(G_i)$, $i=1,2$  respectively, we have, $|V_n( G_1 ) - V_n (G_2)| \lesssim | E_1 \Delta E_2 | /n$. 

Consider first the Poisson Cloning hypergraph $\Pc{ d- C\sqrt{d}\log d}$ \cite{Kim} constructed as follows. 
Let $U_1, \cdots, U_n$ be i.i.d. $\dPois(\frac{(p-1)! (d- C\sqrt{d}\log d)}{n^{p-1}} {n-1 \choose p-1})$ random variables and consider the set of clones 
\begin{align}
\mathscr{L}_1 = \{ (i,j): 1 \leq i \leq n , 1 \leq j \leq U_i \}. \nonumber
\end{align}
Set $\mathbf{U} = \sum_{i=1}^{n} U_i := l_1 p + r_1$, for some $ l _1 \geq 0$, and $ 0 \leq r_1 < p$. Consider the set of half-edges 
\begin{align}
\mathscr{C}_1 = \{(a, k) : 1 \leq a \leq l_1, 1 \leq k \leq p\} \cup \{ (l_1 +1 , k) : 1\leq k \leq r_1 \}. \nonumber
\end{align}
Given $\mathscr{L}_1$ and $\mathscr{C}_1$, let $\upsilon_1$, let $\upsilon_1$ be a uniformly random matching $\upsilon_1 : \mathscr{C}_1 \to \mathscr{L}_1$. The multi hypergraph $\Pc{ d- C\sqrt{d}\log d}$ has vertex set $V = \{1, \cdots, n\}$ and hyperedges $E= \{ \{\upsilon_1((a,k)): 1 \leq k \leq p \} : 1 \leq a \leq l_1 \} \} \cup \{ \{ \upsilon_1((l_1 +1, k )): 1 \leq k \leq r_1 \} \}$. \cite[Theorem 1.1]{Kim} establishes that this model is contiguous to the \ER hypergraph with edge probabilities $(p-1)!(d- C\sqrt{d}\log d )/ n^{p-1}$, and thus it suffices to compare the {\tiny\sf{RHS}} of \eqref{eq:regular_exp_new} to that on the Poisson cloning model. The proof is then complete, by appealing to Lemma \ref{lemma:contiguity}. 


 To facilitate the comparison between $G_2$ and $\Pc$, we use the intermediate hypergraph  $G^{{\sf int}}$, constructed as follows.  Let $W_1, \cdots, W_n$ be i.i.d. $\dPois(d- C\sqrt{d}\log d)$, and similar to the construction of $\Pc$, consider the set of clones 
\begin{align}
\mathscr{L}_2 = \{ (i,j): 1 \leq i \leq n, 1 \leq j \leq W_i \}. \nonumber 
\end{align}
Let $\sum_i W_i = l_2 p + r_2$, for some $l_2 \geq 0$, $r_2 <p$, and consider the set of half-edges 
\begin{align}
\mathscr{C}_2 = \{ (a,k) : 1 \leq a \leq l_2, 1 \leq k \leq p\} \cup \{ (l_2 +1, k): 1 \leq k \leq r_2 \}. \nonumber 
\end{align}
Similar to the construction of $\Pc$, we let $\upsilon_2$ be a uniform random matching $\upsilon_2: \mathscr{C}_2 \to \mathscr{L}_2$. We define $G^{{\sf int}}$ to be a multi hypergraph with vertex set $V= \{1, \cdots, n\}$ and hyperedges $E =  \{ \{\upsilon_2((a,k)): 1 \leq k \leq p \} : 1 \leq a \leq l_2 \} \} \cup \{ \{ \upsilon_2((l_2 +1, k )): 1 \leq k \leq r_2 \} \}$. 

To compare $G^{{\sf int}}$ and $\Pc$,note that $U_i$ is stochastically smaller than $W_i$, and therefore, we can couple the $(U_i, W_i)$ pairs such that $U_i \leq W_i$ for all $1 \leq i \leq n$. We use a two-stage construction, similar to that outlined for the configuration model, to couple $G^{{\sf int}}$ and $\Pc$. To this end, we color the clones $\{(i , j): 1 \leq i \leq n, 1 \leq j \leq U_i \}$ with the color \RED, while the remaining clones are colored \BLUE. The multi hypergraph obtained by the matching $\upsilon_2 : \mathscr{C}_2 \to \mathscr{L}_2$ obtains the graph $G^{{\sf int}}$. Now, we delete all half-edges $a$ such that $\{ \upsilon_2((a,k)) \}$ has at least one \BLUE\, clone. Finally, we add extra half-edges to match the \RED\, clones which have been left un-matched by the deletion procedure. Using Lemma \ref{lemma:matching}, we immediately observe that the graph obtained is distributed as $\Pc{d- C\sqrt{d}\log d }$. This coupling of 
 $\Pc{d- C\sqrt{d}\log d } $ and $G^{{\sf int}}$ ensures that $\E[|E(G^{{\sf clon}}) \Delta E(G^{{\sf int}})|] = O_d(1)$. Thus $\E\sum_{i_1, \cdots, i_p} | A_{G^{{\sf clon}}}(i_1, \cdots, i_p)- A_{G^{{\sf int }}}(i_1, \cdots, i_p) | \lesssim O_d(1)$ implying that for our purposes, we can restrict ourselves to $G^{\sf int}$. 
 
Next, we use the same two stage construction to couple $G^{\sf int }$ and $G_2$. Note that the construction of $G_2$ and $G^{\sf{int}}$ are very similar, and the number of clones in each model can be coupled exactly as long as $\dPois(d- C\sqrt{d}\log d) \leq d$. We choose $C>0$ sufficiently large and note that in this case, the coupling produces hypergraphs which differ in $n o_d(\sqrt{d})$ hyperedges (here we use the normal approximation to the Poisson for $d$ large). Thus, setting $V_n(G^{{\sf clon}})$ to be the value in \eqref{eq:regular_exp_new} with $A_{G^{{\sf clon}}}$ instead of $A_{G_2}$, we see that with high probability as $n \to \infty$, $|V_n^{\R}- V_n(G^{{\sf clon }})| = o_d(\sqrt{d})$. 
Finally, we appeal to the contiguity of the \ER hypergraph and the Poisson cloning model hypergraph ensembles \cite[Theorem 1.1]{Kim} to conclude that Lemma \ref{lemma:contiguity} holds for $V_n(G^{{\sf clon}})$, and thus immediately implies the desired result.

It remains to prove Lemma \ref{lemma:regular_intermediate}. 

\textit{Proof of Lemma \ref{lemma:regular_intermediate}:}
\label{sec:regular_intermediate_proof}
We note that,
\begin{align}
A_{G_1} = A_{G_2} + M, \label{eq:decomp}
\end{align}
where $M= A_{\Gb} - A_{\tGr}$. Thus we have, using \eqref{eq:opt} and \eqref{eq:decomp},
\begin{align}
V_n^{\R} &= \frac{1}{n} \max_{\us \in A_n} \Big\{ \sum_{i_1,\cdots , i_p} A_{G_2}(i_1,\cdots , i_p) f(\sigma_{i_1} , \cdots , \sigma_{i_p}) + \sum_{i_1 , \cdots , i_p} \E[M(i_1,\cdots, i_p)]f(\sigma_{i_1}, \cdots , \sigma_{i_p}) \nonumber\\
&\hspace{100pt}+ \sum_{i_1, \cdots , i_p} M^{\cen}(i_1, \cdots, i_p) f(\sigma_{i_1} , \cdots, \sigma_{i_p}) \Big\}, 
\label{eq:regular_rep1}
\end{align}
where $M^{\cen} = M - \E[M]$.
We claim that 
\begin{align}
\frac{1}{n}\sup_{\us \in A_n} \Big| \sum_{i_1,\cdots,i_p} \Big[M^\cen(i_1,\cdots,i_p) - \frac{p}{n^{p-1}} (Z_{i_1} - \E[Z_{i_1}])  \Big]f(\sigma_{i_1},\cdots, \sigma_{i_p})  \Big| = o_d(\sqrt{d}). 
\label{eq:regular_lemma_intermediate}
\end{align}
We use \eqref{eq:decomp}, \eqref{eq:regular_rep1} and \eqref{eq:regular_lemma_intermediate}, along with  the observation that $\E[A_{G_2}] + \E[M] = \E[A_{G_1}] = d/n^{p-1} + o(1)$, to conclude
\begin{align}
V_n^{\R} = \frac{1}{n} \max_{\us \in A_n} \Big[ &\frac{d}{n^{p-1}}\sum_{i_1,i_2, \cdots, i_p} f(\sigma_{i_1},\cdots, \sigma_{i_p}) + \sum_{i_1,\cdots, i_p} A_{G_2}^{\cen} (i_1,\cdots, i_p) f(\sigma_{i_1},\cdots, \sigma_{i_p}) \nonumber\\
&+ \frac{p}{n^{p-1}} \sum_{i_1,\cdots, i_p} (Z_{i_1} - \E[Z_{i_1}]) f(\sigma_{i_1}, \cdots, \sigma_{i_p})  \Big] + o_d(\sqrt{d}). \label{eq:int_representation}
\end{align}  
To complete the proof, we note that 
\begin{align}
Z_i = d -  \sum_{ l_2, \cdots, l_p} A_{G_2}(i, l_2, \cdots, l_p) + o(1). 
\end{align}
Thus
\begin{align}
\sum_{i_1, \cdots, i_p} (Z_{i_1} - \E[Z_{i_1}]) f(\sigma_{i_1}, \cdots, \sigma_{i_p} ) & = - \sum_{i_1, \cdots, i_p} \sum_{l_2, \cdots, l_p}  A_{G_2}^{\cen}(i_1, l_2, \cdots, l_p) f(\sigma_{i_1}, \cdots, \sigma_{i_p}) + o(n) \nonumber\\
&= - \sum_{i_1, l_2, \cdots, l_p} A_{G_2}^{\cen} (i_1, l_2, \cdots, l_p) \sum_{i_2, \cdots, i_p} f(\sigma_{i_1}, \cdots, \sigma_{i_p}) + o(n)  \label{eqn:int_simplification}
\end{align}
Plugging \eqref{eqn:int_simplification} back into \eqref{eq:int_representation} completes the proof. 
%

%
It remains to prove \eqref{eq:regular_lemma_intermediate}. To this end, for each $\us \in A_n$, we have, using \eqref{eq:regular_rep1}
\begin{align}
\sum_{i_1, \cdots, i_p} M^{\cen}(i_1, \cdots, i_p) f(\sigma_{i_1}, \cdots, \sigma_{i_p}) = 
\sum_{i_1,\cdots, i_p} \Big[A_{\Gb}^\cen(i_1,\cdots, i_p) 
 -  A_{\tGr}^\cen(i_1,\cdots, i_p) \Big] f (\sigma_{i_1}, \cdots, \sigma_{i_p}). \label{eq:M_decomp}
\end{align}
We will handle each term separately and prove
\begin{align}
&\frac{1}{n}\sup_{\us \in A_n} \Big| \sum_{i_1,\cdots,i_p} \Big[A_{\Gb}^\cen(i_1,\cdots,i_p) - \frac{p}{n^{p-1}} (Z_{i_1} - \E[Z_{i_1}])  \Big]f(\sigma_{i_1},\cdots, \sigma_{i_p})  \Big| = o_d(\sqrt{d}). \label{eq:regular_claim1} \\
&\frac{1}{n}\sup_{\us \in A_n} \Big| \sum_{i_1,\cdots,i_p} A_{\tGr}^{\cen}(i_1,\cdots,i_p) f(\sigma_{i_1},\cdots, \sigma_{i_p})  \Big| = o_d(\sqrt{d}). \label{eq:regular_claim2}
\end{align}
\noindent
Equations \eqref{eq:regular_claim1} and \eqref{eq:regular_claim2} automatically imply \eqref{eq:regular_lemma_intermediate}. 

\textit{Proof of \eqref{eq:regular_claim1}:}
We note,
\begin{align}
\sum_{i_1,\cdots, i_p} A_{\Gb}^{\cen}(i_1,\cdots, i_p) f(\sigma_{i_1},\cdots, \sigma_{i_p}) 
&= \sum_{i_1,\cdots, i_p} (A_{\Gb}(i_1,\cdots, i_p) - \E[A_{\Gb}(i_1, \cdots, i_p) | \Deg]) f(\sigma_{i_1} , \cdots, \sigma_{i_p}) \nonumber\\
+  \sum_{i_1,\cdots, i_p} &(\E[A_{\Gb}(i_1,\cdots, i_p) | \Deg] - \E[A_{\Gb}(i_1, \cdots, i_p)]) f(\sigma_{i_1} , \cdots, \sigma_{i_p}), 
\label{eq:regular1}
\end{align}
where $\Deg = (Z_1, \cdots, Z_n)$ is the vector of \BLUE\, clones. Let $\mathcal{E}= \{ \sum_i Z_i \leq n\E[Z_1] + C\sqrt{n \log n} \}$ for some constant $C>0$ suitably large, such that $\P[\mathcal{E}^c]=o(1)$. 
\begin{align}
&\P \Big[\frac{1}{n} \sup_{\us \in A_n} \Big|\sum_{i_1,\cdots, i_p} (A_{\Gb}(i_1,\cdots, i_p) - \E[A_{\Gb}(i_1, \cdots, i_p) | \Deg]) f(\sigma_{i_1} , \cdots, \sigma_{i_p})  \Big| > \Delta \Big] \nonumber\\
&\leq q^n \max_{\us \in A_n} \E\Big[\mathbf{1}_{\mathcal{E}} \P\Big[\frac{1}{n}\Big|\sum_{i_1,\cdots, i_p} (A_{\Gb}(i_1,\cdots, i_p) - \E[A_{\Gb}(i_1, \cdots, i_p) | \Deg]) f(\sigma_{i_1} , \cdots, \sigma_{i_k})  \Big| > \Delta |\Deg\Big] \Big]  + o(1) \nonumber\\
&\leq 2q^n \exp\Big( - \frac{1}{2C_0^2} \,\frac{n^2 \Delta^2}{n\E[Z_1] + C\sqrt{n \log n}} \Big) + o(1), \label{eq:azuma_1}
\end{align}
where \eqref{eq:azuma_1} is derived as follows. Given $\Deg$, we enumerate the $nd$ clones such that the \BLUE\, clones have the smallest 
values. Now, we form the hypergraph sequentially, where at each step, we choose the clone of smallest value and match it to $(p-1)$ randomly chosen unmatched clones. This gives us the natural filtration $\mathscr{F}_0 \subset \mathscr{F}_1 \subset \cdots \subset \mathscr{F}_{nd/p}$, where $\mathscr{F}_0 = \sigma(\Deg)$ and $\mathscr{F}_i$ is the canonical sigma algebra formed after exposing the first $i$ hyperedges. Now, we form the Doob martingale $Z_k= \E[\sum_{i_1, \cdots, i_p } A_{\Gb}(i_1, \cdots, i_p) f(\sigma_{i_1}, \cdots, \sigma_{i_p}) | \mathscr{F}_k]$ such that $Z_0 = \E[ \sum_{i_1, \cdots, i_p } A_{\Gb}(i_1, \cdots, i_p) f(\sigma_{i_1}, \cdots, \sigma_{i_p}) | \Deg]$. We note that setting $n_0 = n \E[Z_1] + C\sqrt{n \log n}$, we have,
$Z_{n_0} = \sum_{i_1, \cdots, i_p } A_{\Gb}(i_1, \cdots, i_p) f(\sigma_{i_1}, \cdots, \sigma_{i_p})$. \eqref{eq:azuma_1} now follows using Azuma-Hoeffding inequality, provided we establish that there exists a constant $C_0$ such that $|Z_k - Z_{k-1}| \leq C_0$ a.s. This follows from straight forward adaptation of the argument in \cite[Theorem 2.19]{wormald} and establishes that 
this probability is $o(1)$ for $\Delta = o_d(\sqrt{d})$. This allows us to neglect this term in \eqref{eq:regular1}. 
To control the second term in \eqref{eq:regular1}, we note that 
\begin{align}
A_{\Gb}(i_1, \cdots, i_p) = A_{\Gb,1}(i_1, \cdots, i_p) + \sum_{j=2}^{p} A_{\Gb,j}(i_1, \cdots, i_p), \label{eq:Gb_decomp}
\end{align}
where $(p-1)! A_{\Gb,j}(i_1, \cdots, i_p)$ counts the number of hyperedges on vertices $i_1, \cdots , i_p$ with exactly $j$ \BLUE\, clones. We show that the contribution due to $\{A_{\Gb,j}: j \geq 2\}$ can be neglected for our purposes. Define $\mathcal{S}_j(i_1, \cdots, i_p) = \{ S \subset \{i_1 , \cdots , i_p\}, |S| = j \}$, where $\{i_1, \cdots, i_p\}$ should be interpreted as a multiset. Therefore, we have, for $i_1, \cdots, i_p$ distinct, 
\begin{align}
\E[A_{\Gb, j}(i_1, \cdots, i_p)  | \Deg] &= \sum_{S \in \mathcal{S}_j(i_1, \cdots, i_p)}  \frac{\prod_{l \in S} Z_l \prod_{l \in S^c} (d- Z_l) } {(nd-1)(nd-2) \cdots (nd-p+1)}, \nonumber
\end{align}
This implies
\begin{align}
\sum_{i_1,\cdots, i_p} \E[A_{\Gb, j}(i_1, \cdots, i_p)  | \Deg] 
&\leq C(p,j ) \frac{ (\sum_i Z_i )^j (\sum_i (d - Z_i))^{p-j}} {(nd)^{p-1}} + o(n),
\end{align}
where $C(p,j)$ is a universal constant dependent on $p,j$ and independent of $n$. Therefore, on the event $\mathcal{E}$,

\begin{align}
&\frac{1}{n}\sup_{\us \in A_n} \Big| \sum_{j=2}^{p} \sum_{i_1,\cdots, i_p} \Big(\E[A_{\Gb,j}(i_1,\cdots , i_p) | \Deg] - \E[A_{\Gb,j}(i_1, \cdots , i_p)] \Big) f(\sigma_{i_1} , \cdots, \sigma_{i_p})  \Big| \nonumber \\
&\leq \frac{ \| f \|_{\infty}}{n} \sum_{j=2}^{p} \sum_{i_1,\cdots, i_p} \Big[\E[A_{\Gb,j}(i_1,\cdots , i_p) | \Deg] + \E[\E[A_{\Gb,j}(i_1,\cdots , i_p) | \Deg] ]  \Big] \nonumber\\
&\leq \|f \|_{\infty} \sum_{j=2}^{p} C(p,j) \Big[\frac{(\E[Z_1])^{j}}{d^{j-1}} \Big]  + o(1)= o_d(\sqrt{d})\nonumber 
\end{align}
for $j\geq 2$. This allows us to neglect this term in \eqref{eq:regular_claim1}. Finally, we are left with the first term in \eqref{eq:Gb_decomp}. We have,
\begin{align}
\E[A_{\Gb,1}(i_1, \cdots, i_p) | \Deg] =  \sum_{j=1}^{p} \frac{Z_{i_j}\prod_{l\neq j}(d- Z_{i_l})}{(nd-1)\cdots (nd - p +1)} \label{eq:regular3}
\end{align}
Therefore, using the law of large numbers, we have, with high probability, 
\begin{align}
&\frac{1}{n} \sum_{i_1, \cdots, i_p} ( \E[ A_{\Gb,1}(i_1, \cdots, i_p) | \Deg] - \E[ A_{\Gb,1}(i_1, \cdots, i_p)] ) f(\sigma_{i_1}, \cdots, \sigma_{i_p}) \nonumber \\
& = \frac{1}{n^p} \sum_{i_1, \cdots, i_p} \sum_{j=1}^{p} (Z_{i_j} - \E[Z_{i_j}]) f (\sigma_{i_1},  \cdots, \sigma_{i_p}) + o_d(\sqrt{d}). 
\end{align}
This completes the proof of \eqref{eq:regular_claim1}.


\textit{Proof of \eqref{eq:regular_claim2}:}
 To this end, setting $\Degsec = (Y_1, \cdots, Y_n)$ as the number of  \RED\, clones of the $i^{th}$ vertex free at the second stage of matching, we have, as in \eqref{eq:regular1},
\begin{align}
\sum_{i_1, \cdots, i_p}  A_{\tGr}^{\cen}(i_1,  \cdots , i_p)  f(\sigma_{i_1}, \cdots, \sigma_{i_p}) 
&= \sum_{i_1, \cdots, i_p} ( A_{\tGr}(i_1,  \cdots , i_p) - \E[ A_{\tGr}(i_1, \cdots, i_p ) | \Degsec]) f(\sigma_{i_1}, \cdots, \sigma_{i_p}) + \nonumber\\
+ \sum_{i_1, \cdots, i_p}& ( \E[A_{\tGr}(i_1,  \cdots , i_p)|\Degsec] - \E[ A_{\tGr}(i_1, \cdots, i_p ) ]) f(\sigma_{i_1}, \cdots, \sigma_{i_p}).  \label{eq:tGr_decomp}
\end{align}
The contribution due to the first term in \eqref{eq:tGr_decomp} can be shown to be $o_d(\sqrt{d})$ similar to \eqref{eq:azuma_1} using Azuma's inequality as under the event $\mathcal{E}$, $\sum_{i} Y_i \leq (p-1) ( n \sum_i \E[Z_i] + C \sqrt{n \log n})$. Next, we will control the second term in \eqref{eq:tGr_decomp}. 
To this end, we have,
\begin{align}
\E[A_{\tGr}(i_1, \cdots, i_p) | \Degsec] = \frac{Y_{i_1}\cdots Y_{i_p}}{(\sum_i Y_i -1) \cdots (\sum_i Y_i - p +1)} =\frac{ \prod_{j=1}^{p} (\E[Y_{i_j} | \Deg] + \ve_{i_j}) }{ \prod_{j=1}^{p-1}(\sum_i Y_i - j )  }, \label{eq:regular-condexp}
\end{align}
where we express $Y_i = \E[Y_i | \Deg] + \ve_i$ for $1\leq i \leq n$. 
%
We will establish that there exists $\Delta= o_d(\sqrt{d})$ such that 
\begin{align}
\P\Big[\frac{1}{n} \sum_{i_1, \cdots , i_p} \Big| \E[A_{\tGr}(i_1, \cdots, i_p)|\Degsec] - \E[A_{\tGr}(i_1,\cdots, i_p)] \Big| > \Delta \Big] = o(1). \label{eq:regular-claim2-subclaim}
\end{align} 
Consider the event $\mathcal{E}_1= \{ (p-1)\sum_i Z_i- nC_0 p(p-1)^2 \log d / \sqrt{d} \leq \sum_i Y_i \leq (p-1) \sum_i Z_i  \} \cap \mathcal{E}$, for some constant $C_0>0$ sufficiently large. Given the counts of \BLUE\, clones $\Deg= (Z_1, \cdots, Z_n)$, the probability that a hyper-edge contains at least 
two \BLUE\, clones is less than $\P[\dBin(p-1, \sum_i Z_i / (n d - \sum_i Z_i)) >0] := p(\Deg)$. Thus the number of hyper-edges in $G_1$ having at least two \BLUE\, clones is stochastically dominated by $\dBin(\sum_i Z_i , p(\Deg))$. We note that with high probability, $\dBin(\sum_i Z_i , p(\Deg)) \leq nC_0 (p-1) \log d / \sqrt{d}$ with high probability for some constant $C_0>0$. Finally, each hyperedge contains at most $p$ \BLUE\, clones. This implies $\P[\mathcal{E}_1 ] = 1- o(1)$ as $n \to \infty$. Thus we have, setting
\begin{align}
\Xi = \frac{1}{n} \sum_{i_1, \cdots, i_p} \Big| \E[A_{\tGr}(i_1, \cdots, i_p)|\Degsec] - \E[A_{\tGr}(i_1,\cdots, i_p)] \Big| , \nonumber
\end{align}
$\P[|\Xi | > \Delta] \leq \P[ |\Xi | > \Delta, \mathcal{E}_1] + o(1)$.
Therefore, it suffices to prove that $\P[ |\Xi |  > \Delta , \mathcal{E}_1] =o(1)$ for some $\Delta = o_d(\sqrt{d})$ chosen suitably. 
 Henceforth in the proof, for conciseness of notation, we will use $[a]_{p-1} = (a-1) \cdots (a-p+1)$.  From \eqref{eq:regular-condexp}, we have, setting $\delta(d) = C_0 p (p-1)^2 \frac{\log d}{\sqrt{d}}$,
\begin{align}
 \frac{Y_{i_1}Y_{i_2}\cdots Y_{i_p}}{[(p-1)n \E[Z_1]+ o(n)]_{p-1}} \bone_{\mathcal{E}_1}    \leq  \E[A_{\tGr}(i_1, \cdots, i_p) | \Degsec] \bone_{\mathcal{E}_1} \leq 
 \frac{Y_{i_1} \cdots Y_{i_p} }{ [(p-1) n \E[Z_1] - n \delta(d)]_{p-1}} \bone_{\mathcal{E}_1}. \nonumber
\end{align}
We will establish that on the event $\mathcal{E}_1$, 
\begin{align}
\frac{1}{n}\sum_{i_1, \cdots, i_p} \Big|\frac{Y_{i_1}Y_{i_2}\cdots Y_{i_p}}{[(p-1)n \E[Z_1]+ o(n)]_{p-1}} - \frac{Y_{i_1} \cdots Y_{i_p} }{ [(p-1) n \E[Z_1] - n \delta(d)]_{p-1}}   \Big| \leq \Delta_1 \label{eq:regular_finalclaim1}
\end{align}
for some $\Delta_1 = o_d(\sqrt{d})$. Further, we show that on the event $\mathcal{E}_1$,
\begin{align}
\Xi_1 = \frac{1}{n} \sum_{i_1, \cdots, i_p } \Big| \frac{Y_{i_1}Y_{i_2}\cdots Y_{i_p}}{[(p-1)n \E[Z_1]+ o(n)]_{p-1}} - \E\Big[\frac{Y_{i_1}Y_{i_2}\cdots Y_{i_p}}{[(p-1)n \E[Z_1]+ o(n)]_{p-1}} \bone_{\mathcal{E}_1} \Big]\Big| \leq \Delta_2 \label{eq:regular_finalclaim2}
\end{align}
for some $\Delta_2 = o_d(\sqrt{d})$. We note that triangle inequality along with 
\eqref{eq:regular_finalclaim1} and \eqref{eq:regular_finalclaim2} implies that  on the event $\mathcal{E}_1$, $\Xi \leq 2\Delta_1 + \Delta_2 + 1/n \sum_{i_1, \cdots, i_p} \E[ A_{\tGr}(i_1, \cdots, i_p) \bone_{\mathcal{E}_1^c}]$, and thus  \eqref{eq:regular-claim2-subclaim} follows, provided we establish that 
\begin{align}
\frac{1}{n} \sum_{i_1, \cdots, i_p} \E[A_{\tGr}(i_1,\cdots,i_p) \bone_{\mathcal{E}_1^c}] \to 0. \label{eq:regular-finalclaim}
\end{align}
\eqref{eq:regular-finalclaim} is derived as follows.
\begin{align}
\frac{1}{n} \sum_{i_1, \cdots, i_p} \E[A_{\tGr}(i_1,\cdots,i_p) \bone_{\mathcal{E}_1^c}] = \frac{1}{n} \E \Big[ \frac{(\sum_i Y_i)^p} { [\sum_i Y_i ]_{p-1}} \bone_{\mathcal{E}_1^c} \Big] 
\leq d \,\E\Big[\Big(\frac{\sum_i Y_i}{\sum_i Y_i -p +1} \Big)^{2p-2}\Big] \P[\mathcal{E}_1^c] \lesssim \P[\mathcal{E}_1^c], \nonumber 
\end{align}
where the inequalities are respectively derived by using Cauchy-Schwarz inequality and that $\sum_i Y_i /(\sum_i Y_i -p+1) \leq p$ as $\sum_i Y_i \geq p$.

Thus it remains to bound 
\eqref{eq:regular_finalclaim1} and \eqref{eq:regular_finalclaim2}. We first establish \eqref{eq:regular_finalclaim1}. On the event $\mathcal{E}_1$, 
\begin{align}
&\frac{1}{n} \sum_{i_1,\cdots, i_p} \frac{Y_{i_1} \cdots Y_{i_p} }{[(p-1)n \E[Z_1] + o(n)]_{p-1}} \Big| 1 - \frac{[(p-1)n \E[Z_1] + o(n)]_{p-1}}{[(p-1)n \E[Z_1] - n\delta(d) ]_{p-1}} \Big| \nonumber\\
&\leq \frac{1}{n} (n(p-1) \E[Z_1] + o(n) ) (1+ o(1)) \Big| 1 - \frac{[(p-1)n \E[Z_1] + o(n)]_{p-1}}{[(p-1)n \E[Z_1] - n\delta(d) ]_{p-1}} \Big| \lesssim (p-1) \delta(d) = o_d(\sqrt{d}).\nonumber
\end{align}
Finally, we come to \eqref{eq:regular_finalclaim2}. We express $Y_i = \sum_{j=1}^{d- Z_i} \bone_j(i)$, where 
$\bone_j(i) =1 $ if and only if the $j^{th}$ \RED\, clone of vertex $i$ is included in an hyper-edge with at least one 
\BLUE\, clone. Thus we have, by symmetry, that $\E[Y_i | \Deg] = (d- Z_i) p_0(\Deg)$, where $p_0(\Deg) = \P[ \bone_1(i) = 1 | \Deg]$. We note that, 
\begin{align}
p_0(\Deg) = 1- \frac{[nd-\sum_i Z_i  ]_{p-1}}{[nd]_{p-1}} .  
\label{eq:p0_asymptotic}
\end{align}
Next, we have, 
\begin{align}
\Var(Y_i | \Deg) &= \sum_{ 1 \leq j  \leq d-Z_i} \Var(\bone_j | \Deg) + \sum_{1\leq j \neq j' \leq d-Z_i} \Cov(\bone_j (i) , \bone_j' (i) | \Deg) \nonumber\\
&= (d-Z_i) p_0(\Deg) ( 1- p_0(\Deg)) + (d- Z_i ) (d- Z_i - 1) (q_0 (\Deg) - p_0(\Deg)^2), 
\end{align}
where we set $q_0(\Deg) = \P[ \bone_1 = \bone_2 =1 | \Deg]$. We note that, on the event $\mathcal{E}_1$, using \eqref{eq:p0_asymptotic},   
\begin{align}
q_0(\Deg) &= 2p_0(\Deg) - 1  + \frac{[nd-\sum_i Z_i -1 ]_{2(p-1)}}{ [nd-1]_{2(p-1)}} + o(1) = p_0(\Deg)^2 + O(1/n).   \label{eq:regular-indep}
\end{align}
This proves that $\Var(Y_i | \Deg) = \Theta(\E[Z_1])$ and therefore, by Cauchy-Schwarz inequality, on the event $\mathcal{E}_1$, 
$\frac{1}{n} \sum_{i=1}^n \E[ |\ve_i| | \Deg] \leq u(d) = o_d(\sqrt{d})$,
where $Y_i = \E[Y_i | \Deg] + \ve_i$. To control $\Var(\sum_i |\ve_i| | \Deg)$, note that a calculation similar to \eqref{eq:regular-indep} proves that the indicators $\{ \bone_j (i) : 1\leq i \leq n , 1 \leq j \leq d-Z_i \}$ are approximately independent, which enforces that $\Var(\sum_i |\ve_i| |\Deg) = \Theta(n)$. 
Thus we have, for some $\Delta_3 > 2u(d) $, using Chebychev inequality,
\begin{align}
\P\Big[ \frac{1}{n} \sum_i |\ve_i | > \Delta_3, \mathcal{E}_1 \Big] \leq \E\Big[\bone_{\mathcal{E}} \frac{\Var(\sum_i |\ve_i| |\Deg)}{n^2 u(d)^2} \Big] = o(1). \nonumber 
\end{align}
To establish \eqref{eq:regular_finalclaim2}, we start by decomposing 
\begin{align}
Y_{i_1} \cdots Y_{i_p} &= ( (d-Z_{i_1})p_0(\Deg) + \ve_{i_1} ) \cdots ( (d-Z_{i_p}) p_0(\Deg) + \ve_{i_p}) \nonumber\\
&:=  (p_0(\Deg))^p\prod_{ 1\leq j \leq p} (d-Z_{i_j}) + T_1 = d^p p_0(\Deg)^p + T_{11} + T_1. \label{eq:regular_decomp}
\end{align}
The contribution due to the first term cancels due to the Law of Large numbers. For any term in $T_{11}$, we have, on the event $\mathcal{E}_1$, for some $j \geq 1$, the contribution is bounded by 
\begin{align}
\frac{d^{p-j} p_0(\Deg)^p (\frac{1}{n} \sum_i Z_i )^j}{ (\E[Z_1])^{p-1}} \leq \frac{ (\E[Z_1])^{j+1}} { d^j}= o_d(\sqrt{d}) . \nonumber
\end{align}

Finally, the contribution due to any term in $T_1$ is dominated as follows. Fix an integer $j \leq p-1$. Then we have an upper bound of the form
\begin{align}
\frac{(\E[Z_1])^{p-j} (\frac{1}{n} \sum_i |\ve_i | )^{j} }{ (\E[Z_1])^{p-1}} = o_d(\sqrt{d}) \nonumber 
\end{align}
whenever $j \geq 1$. This completes the proof. 

Finally, we finish the section by establishing the validity of the two-step construction used frequently in our argument. 
\begin{lemma}
\label{lemma:matching}
Consider $N= m+n$ labelled balls, $m$-\RED\, and $n$-\BLUE. Assume that $p| m,n$. The following two-step procedure obtains a uniform random partition of the $m$-\RED\, balls into groups of size $p$. 
\begin{itemize}
\item[1.] Group the $N$ balls at random into blocks of size $p$.

\item[2.] Remove the \BLUE\, balls and re-match the \RED\, balls left unmatched as a result into groups of size $p$. 
\end{itemize}
\end{lemma}

\textit{Proof of Lemma \ref{lemma:matching}:}
Let $S_{N,p} = \frac{N!}{(p!)^{N/p} (N/p)!}$ denote the total number of ways to partition the $N$-balls into groups of size $p$. Given a fixed partition $P$ of the \RED\, balls, it can be obtained using the two step procedure as follows. We choose $s$ groups formed at the first stage of matching, while the remaining $(m/p-s)$ groups are formed at the second stage. This implies, setting $\mathscr{P}$ to be random matching obtained by the two-step procedure,
\begin{align}
\P[\mathscr{P} = P] = \sum_{s} { m/p \choose s} \frac{A(s)}{S_{N,p}} \frac{(p!)^{m/p-s} (m/p -s)! }{(m-ps)!}= \frac{(p!)^{m/p} (m/p)!}{m!} \frac{1}{S_{N,p}} \sum_s A(s) {m \choose ps} \frac{(ps)!}{(p!)^s s!} , \nonumber
\end{align}
where $A(s)$ denotes the number of ways to group the remaining $N-ps$ balls into $p$-size groups such that no group has all \RED\, balls. Now, we note that any random partition of the $N$ balls can be obtained by first choosing $ps$ \RED\, balls to be grouped among themselves and subsequently grouping the remaining balls such that no block has all \RED\, balls. This implies 
\begin{align}
\frac{1}{S_{N,p}} \sum_s A(s) {m \choose ps} \frac{(ps)!}{(p!)^s s!} = 1, \nonumber
\end{align}
which immediately completes the proof.


\section{Proof of Proposition \ref{prop:equivalence}}
\label{section:proof-prop}
For computing the Gaussian surrogate in the \ER case, one can sum over all unrestricted tuples $\{(i_1, \cdots, i_p): 1 \leq i_1 , \cdots, i_p \leq n\}$ in \eqref{eq:gaussian_opt}, with $\{J_{i_1, \cdots, i_p} \}$ being the standard symmetric $p$-tensor, at the cost of incurring a $o(1)$ error. For notational convenience, we assume in this Section that we indeed work with this slightly modified value. Recall that the \ER case corresponds to the specific choice of the kernels $\kappa_1 = 1$ and $\kappa_2 = \frac{1}{(p-1)!}$. Moreover, we recall that a sequence of random variables $\{X_n\}= o_d(1)$ if there exists a deterministic function $f(d) = o(1)$ as $d \to \infty$ such that $\P[ |X_n| \leq f(d) ] \to 1$ as $n \to \infty$.

 Note that \eqref{eq:gaussian_opt} and \eqref{eq:regular_opt} imply that with high probability as $n \to \infty$, 
\begin{align}
|V_n - V_n^{\R}| \leq p \sqrt{d}  \E\Big[ \max_{\us \in A_n} \Big| \frac{1}{n^p} \sum_{i_1, \cdots, i_p} G_{i_1} f(\sigma_{i_1}, \cdots, \sigma_{i_p}) \Big| \Big] + o_d(\sqrt{d}). \label{eq:int10}
\end{align}

The proof for sufficiency of (C1) is comparatively straightforward.
Indeed, we have, for each $\us \in A_n$,
\begin{align}
\frac{1}{n^{p-1}} \sum_{i_2 , \cdots, i_p} f(\sigma_{i_1}, \cdots , \sigma_{i_p}) =\sum_{k_2,\cdots, k_p} f(\sigma_{i_1}, k_2, \cdots, k_p) m_{k_2}(\us) \cdots m_{k_p}(\us)  = \eta + r_{\sigma_{i_1}}(\us).  \nonumber
\end{align}
Plugging this expression into \eqref{eq:int10}, we have, 
\begin{align}
\frac{1}{n^p} \sum_{i_1, \cdots, i_p} G_{i_1} f(\sigma_{i_1}, \cdots, \sigma_{i_p})= \frac{\eta}{n} \sum_{i_1} G_{i_1} + \frac{1}{n} \sum_{i_1} G_{i_1} r_{i_1}(\us).   \nonumber
\end{align}

Therefore, we have, under (C1) 
\begin{align}
 \E\Big[ \max_{\us \in A_n} \Big| \frac{1}{n^p} \sum_{i_1, \cdots, i_p} G_{i_1, \cdots, i_p} f(\sigma_{i_1}, \cdots, \sigma_{i_p}) \Big| \Big] \leq |\eta| \E\Big[ \frac{1}{n} \Big|\sum_{i_1} G_i \Big| \Big] + o_d(1) \E\Big[ \frac{1}{n} \sum_{i_1} |G_{i_1}| \Big]
\end{align}

It remains to bound each of these terms.
Note that marginally,  each $G_i \sim \dN(0, \sigma^2)$ for some $\sigma^2 =O(1)$  and $\cov(G_i , G_j) = O(1/n)$. It follows that $\Var(\sum_{i_1} G_{i_1}) = O(n)$ and thus $\E [ |  \sum_{i_1}G_{i_1}/n | ] \to 0$. This controls the first term.To control the second term, observe that  $\E[\sum_i |G_i | ] = n \sigma \sqrt{2/\pi}$ and thus the term is $o_d(1)$.  
%
%
%
%
%
%

It remains to check the validity of the thesis under the condition (C2). Without loss of generality, we can and will assume that $\mathcal{X} = \{1 ,\cdots , q \}$ for some $q \geq 1$. 
Applying Theorem \ref{lemma:comparison}, we have,
\begin{align}
V_n = \E[\max_\alpha ( d \alpha +  T_n^{\alpha} \sqrt{d} )] + o_d(\sqrt{d}), \label{eq:max_cut_opt_rep}
\end{align}
where $\alpha, T_n^{\alpha}$ are defined as 
\begin{align}
T_n^{\alpha} = \frac{1}{n}\max_{\us \in [q]^n} \sum_{i_1, i_2, \cdots, i_p} \frac{J_{i_1, \cdots, i_p}}{n^{(p-1)/2}} f(\sigma_{i_1}, \cdots, \sigma_{i_p}), \,\,\,\,
\textrm{subject to } \frac{1}{n^p} \sum_{i_1 , \cdots , i_p} f(\sigma_{i_1}, \cdots, \sigma_{i_p})= \alpha, \label{eq:max_cut_opt1}
\end{align}
and $\mathbf{J} = (J_{i_1 i_2 \cdots i_p})$ is the standard symmetric Gaussian $p$-tensor. Throughout this proof, we set 
\begin{align}
H_f(\us)= \sum_{i_1 , i_2 , \cdots , i_p} \frac{J_{i_1, \cdots, i_p}}{n^{(p-1)/2}} f(\sigma_{i_1}, \cdots, \sigma_{i_p}).  \nonumber
\end{align}
We observe that the collection of random variables $\{ H_f(\us) : \us \in S_n \}$ forms a centered Gaussian process on $q^n$ elements.  The following lemma bounds the expectation and variance of the supremum of this Gaussian process. 
\begin{lem}
\label{lemma:exp_var_sup_gaussian}
There exists universal constants $C_1, C_2 >0$ such that 
\begin{align}
\E\Big[\max_{\us \in [q]^n} H_f(\us) \Big] \leq C_1 n , \,\,\,\,\, \Var\Big( \max_{\us \in [q]^n} H_f(\us) \Big) \leq C_2 n. \nonumber 
\end{align}
\end{lem}

\textit{Proof of Lemma \ref{lemma:exp_var_sup_gaussian}:} Consider any centered Gaussian process $\{Z_s : s \in S\}$ with $|S| < \infty$ and $\Var(Z_s) \leq \tau^2$ for all $s \in S$. Then it is well known that $\E[\max_{s \in S} Z_s] \leq \sqrt{2 \tau^2 \log |S|}$. Further, it is known that $\Var(\max_{s} Z_s ) \leq \max_s \Var(Z_s) \leq \tau^2$ \cite{houdre1995multivariate}. In this case, we have a centered Gaussian process $\{H_f(\us): \us \in [q]^n \}$ such that $|S| = q^n$ and $\Var(H_f(\sigma)) \lesssim n $. This completes the proof.

Given Lemma \ref{lemma:exp_var_sup_gaussian}, we return to proof of Proposition \ref{prop:equivalence}. 
The constraint in \eqref{eq:max_cut_opt1} may be expressed in terms of the empirical distribution of spins of each type. Recall the function $\mathbf{m}$ defined in the Introduction. For each $\mathbf{p}= (p_1,\cdots, p_q)$ with $\mathbf{p} \in {\sf Sim}:= \{ \mathbf{a} : a_j \geq 0, \sum_j a_j  = 1 \}$, we set, 
\begin{align}
\Sigma(\mathbf{p}) = \{\us \in [q]^n: \mathbf{m}(\us) = \mathbf{p}   \}. \nonumber  
\end{align}
{Recall the definition of $\Psi$ from \eqref{eq:defnpsi} and note that the deterministic constraint may be expressed as $\Psi(\mathbf{m}(\us)) = \alpha$. We define 
\begin{align}
T_1(n,d) = \max_{{\bm} \in {\sf Sim}} \Big( d \Psi(\mathbf{m}) + \frac{\sqrt{d}}{n} \max_{\us \in \Sigma(\bm)} H_f(\us) \Big). \nonumber
\end{align}

\noindent
Thus the constrained optimization problem \eqref{eq:max_cut_opt_rep} may be re-expressed as 
\begin{align}
V_n 
= \E[T_1(n,d)] + o_d(\sqrt{d}).  \nonumber
\end{align}
Now, we assume that $\Psi(\mathbf{m} )$ is maximized at $\bm^*$. We define 
\begin{align}
T_2(n,d) &= d  \Psi(\bm^*) +   \frac{\sqrt{d}}{n} \max_{\us \in \Sigma( \bm^*)} H_f(\us).  \nonumber
\end{align}
 Trivially, we have $T_1(n,d) \geq T_2(n,d)$.
Thus we have the lower bound
\begin{align}
\liminf\,\, \E[T_1(n,d)] \geq d \Psi(\bm^*) +  \sqrt{d} \liminf\, \E\Big[ \frac{1}{n} \max_{\us \in \Sigma(\bm^*)} H_f(\us) \Big].  \label{eq:max_cut_lower}
\end{align}
To derive a matching upper bound we proceed as follows. 
 }
 Note that 
\begin{align}
\E[T_1(n,d) - T_2(n,d)] = \E[ (T_1(n,d) - T_2(n,d) ) \bone_\mathcal{F} ] + \E[ (T_1(n,d) - T_2(n,d)) \bone_{\mathcal{F}^c} ]\nonumber
\end{align}
for any event $\mathcal{F}$.  { Using Cauchy-Schwarz inequality, 
\begin{align}
| \E[ (T_1(n,d) - T_2(n,d)) \bone_{\mathcal{F}^c} ]| \leq  \sqrt{ \E[ (T_1(n,d) - T_2(n,d))^2] \P[\mathcal{F}^c]}. \nonumber
\end{align}
The following lemma derives an upper bound on $\E[(T_1(n,d) - T_2(n,d))^2]$. 
\begin{lem}
\label{lem:second_moment_control}
There exists a universal constant $C_1$, depending on $d$, independent of $n$, such that 
\begin{align}
\E[(T_1(n,d) - T_2(n,d))^2] \leq C_1. \nonumber 
\end{align}
\end{lem}
\noindent 
We defer the proof to the end of the section for ease of exposition. }
%
%
%
%
As a result, for any event $\mathcal{F}$ such that $\P[\mathcal{F}^c] \to 0$ as $n\to \infty$, we have
\begin{align}
\E[T_1(n,d) - T_2(n,d)] =  \E[ (T_1(n,d) - T_2(n,d) ) \bone_\mathcal{F} ]  + o(1).  \label{eq:truncation-goodevent}
\end{align} 

Consider the event $\mathcal{F} = \Big \{ \sum_i p_i^* (1- p_i^*) \geq \Psi(\bm^*) - \frac{A}{\sqrt{d}} \Big\}$, where $A>0$ is sufficiently large, to be specified later. We will establish that for $A>0$ sufficiently large, $\P[\mathcal{F}^c] = o(1)$ as $n \to \infty$. To this end, we observe,
\begin{align}
\mathcal{F}^c \cap \Big\{ \max_{\us \in \Sigma(\bm^*)} H_f(\us) \geq - C\sqrt{n \log n} \Big\} \subseteq \Big\{ \max_{\us} H_f(\us) \geq \frac{An}{2} - C\sqrt{n \log n} \Big\}. \nonumber  
\end{align}

Then we have,
\begin{align}
\P[ \mathcal{F}^c ] &\leq \P\Big[   \max_{\us} H_f(\us) \geq \frac{An}{2} - C\sqrt{n \log n}\Big]  
+ \P\Big[\max_{\us \in \Sigma(\bm^*)}H_f(\us)< - C\sqrt{n \log n}   \Big].  \label{eq:maxqcut_bounds}
\end{align}
We show that each term in the {\tiny{\sf RHS }} of \eqref{eq:maxqcut_bounds} is $o(1)$. 
To bound the first term, we note that Lemma \ref{lemma:exp_var_sup_gaussian} implies 
\begin{align}
\E \Big[\max_{\us} H_f(\us)  \Big] \leq  C_1 n,  \,\,\,\,\,
\Var \Big[\max_{\us} H_f(\us) \Big] \leq  C_2 n.  \label{eq:gaussianprocess-bounds}
 \end{align}
 Thus we have, by applying a traditional concentration bound on the suprema of a Gaussian process \cite{blm}, 
 \begin{align}
 \P\Big[   \max_{\us} H_f(\us) \geq \frac{An}{2} - C\sqrt{n \log n}\Big] \leq \exp \Big( - C_0 \frac{(\frac{An}{2} - C\sqrt{n \log n} - \E[ \max_{\us} H_f(\us)])^2}{n} \Big) \nonumber 
 \end{align}
 for some universal constant $C_0>0$. Thus using \eqref{eq:gaussianprocess-bounds}, we note that if $A$ is chosen sufficiently large, the probability decays to zero as $n \to \infty$. 
 
We bound the second term using Chebychev inequality. To this end, we note from \eqref{eq:gaussianprocess-bounds} that the variance of the maxima is $O(n)$ while the expectation of the supremum is non-negative and therefore
\begin{align}
\P\Big[\max_{\us \in \Sigma(\bm^*)}H_f(\us)< - C\sqrt{n \log n}   \Big] \leq \frac{\Var( \max_{\us} H_f(\us) ) }{ C^2 n \log n }= o(1). \nonumber
\end{align}
This ensures that for $A$ chosen sufficiently large, $\P[\mathcal{F}^c] = o(1)$. 
 
 Using \eqref{eq:truncation-goodevent}, we see that it suffices to bound $\E[ (T_1(n,d) - T_2(n,d)) \bone_{\mathcal{F}}]$. 
We set $\us^* = \argmax\, T_1(n,d)$ and the corresponding empirical distribution of spins $\mathbf{p}^*= (p_1^*, \cdots, p_q^*)$.
Further, we fix a function $\psi : [q]^n \to \Sigma(\bm^*)$ which maps every $\us \in [q]^n$ to the configuration set $\Sigma(\bm^*)$ by changing the minimum number of coordinates for each configuration. 
 On the event $\mathcal{F}$, we have, setting $p_i^* = m_i^*+ \ve_i$, $\sum_i \ve_i =0$ and using Taylor's theorem and  $(C2)$, we have, for $d \geq d(f)$, and some $\xi^*$ on the segment joining $\mathbf{p}^*, \bm^*$, 
 \begin{align}
\frac{A}{\sqrt{d}} \geq |\bar{\Psi}(\mathbf{p}^*) - \bar{\Psi}(\bm^*)| = | (\mathbf{p}^* - \bm^*)^{{\sf T}} \grad^2 \bar{\Psi}(\xi^*) (\mathbf{p}^* - \bm^*)|   \implies \| \mathbf{p}^* - \bm^* \|_2^2 \leq \frac{A}{c\sqrt{d}}.
\nonumber  
 \end{align}
 Using Cauchy-Schwarz inequality, we have, $\| \mathbf{p}^* - \bm^* \|_1 \leq \sqrt{A q}/ (\sqrt{c} d^{1/4})$. The $\ell_1$ distance between two discrete probability vectors is equivalent to the TV distance, and thus it is easy to see that there exists a constant $C_0>0$ such that we can re-label at most $n C_0/d^{1/4}$ spins in $\us^*$ to get the configuration $\us^{**} \in \Sigma(\bm^*)$ chosen earlier. On the event $\mathcal{F}$, using the definition of $T_2(n,d)$, we have, 
 \begin{align}
 T_2(n,d) \geq d\Psi(\bm^*) + \frac{\sqrt{d}}{n} H_f(\us^{**}) 
 \geq T_1(n,d) - e, 
 \end{align}
where we set 
$e= \frac{\sqrt{d}}{n} \Big[ H_f(\us^*) - H_f(\us^{**}) \Big]$. 
On the event $\mathcal{F}$, we have, 
\begin{align}
e \leq \frac{\sqrt{d}}{2n} \max_{\{ \us : \| \mathbf{p}(\us) - \Sigma(\bm^*) \|_1 \leq \frac{C_0}{d^{1/4}}\}} \Big[ \sum_{i_1, \cdots , i_p} \frac{J_{i_1 i_2 \cdots i_p}}{n^{(p-1)/2}} \Big( f(\sigma_{i_1}, \cdots, \sigma_{i_p}) - f(\psi(\sigma)_{i_1}, \cdots, \psi(\sigma)_{i_p}) \Big)\Big]. \nonumber
\end{align}
We note that $(f(\sigma_{i_1} , \cdots, \sigma_{i_p}) - f(\psi(\sigma)_{i_1}, \cdots, \psi(\sigma)_{i_p}) \neq 0$ if and only if $(\sigma_{i_1},\cdots, \sigma_{i_p} ) \neq (\psi(\sigma)_{i_1}, \cdots,  \psi(\sigma)_{i_p})$. The number of such terms is bounded by $C_1 n^p/d^{1/4}$ for some constant $C_1$. Thus the variance of each Gaussian is bounded by $ C_1 n/{d^{1/4}}$. Now using the bound on the expected suprema of a Gaussian process described above, we get that the RHS is $o_d(\sqrt{d})$. 
Thus $\E[e \bone_{\mathcal{F}}] = o_d(\sqrt{d})$. This establishes that $\E[T_1(n,d) - T_2(n,d)] = o_d(\sqrt{d})$. 

Now we look at the $d$-regular problem. The same argument goes through in this case and we see that the optimal value may be attained (up to $o(\sqrt{d})$ corrections) by restricting to the configuration space $\Sigma(\bm^*)$. We note that $\Psi(\cdot)$ is a smooth function which is maximized at $\bm^*$. We set up the Lagrangian
\begin{align}
\Xi = \Psi(\bm) - \lambda( \sum_i m_i - 1) - \sum_i \mu_i m_i. \nonumber 
\end{align}
Setting $\grad\Xi =0$, we have, $\grad \Psi(\bm) = \lambda \bone + \mathbf{\mu}$, where $\mathbf{\mu}= (\mu_1, \cdots, \mu_q)$. Using the complimentary slackness conditions, we have $\mu_i=0$ for all $i \leq q$. It is easy to see now that at the stationary point, condition $(C1)$ is satisfied. Thus $(C1)$ is satisfied for all $\us \in \Sigma(\bm^*)$. 
The argument outlined in the first part of the proof now enforces the desired conclusion. 

{
It remains to prove Lemma \ref{lem:second_moment_control}. We will first need a lemma which bounds  the second moment of the supremum of the Gaussian process $\{ H_f(\us): \us \in [q]^n \}$. 
\begin{lem}
\label{lemma:gaussian_process_bound}
There exists a universal constant $C>0$, depending on $f$, such that 
\begin{align}
\E\Big[ \Big( \max_{\us \in [q]^n} H_f(\us) \Big)^2 \Big] \leq C n^2. \nonumber 
\end{align}
\end{lem}

\textit{Proof of Lemma \ref{lemma:gaussian_process_bound}: }
We note that 
\begin{align}
\E\Big[ \Big( \max_{\us \in [q]^n} H_f(\us) \Big)^2 \Big] = \E^2 \Big[  \max_{\us \in [q]^n} H_f(\us) \Big] + \Var \Big(  \max_{\us \in [q]^n}  H_f(\us) \Big). \nonumber
\end{align}
and bound each term on the {{\tiny\sf{RHS}}} using Lemma \ref{lemma:exp_var_sup_gaussian}. 

\noindent
Given Lemma \ref{lemma:gaussian_process_bound}, we complete the proof of Lemma \ref{lem:second_moment_control} next. 

\textit{Proof of Lemma \ref{lem:second_moment_control}:} 
We observe
\begin{align}
\E[(T_1(n,d) - T_2(n,d))^2] \leq 2 \Big[  \E[ T_1(n,d)^2] + \E[T_2(n,d)^2 ]\Big]. \nonumber
\end{align}
We bound each term on the RHS separately. First, observe that $\|\Psi\|_{\infty} \leq \|f \|_{\infty}$ and thus 
\begin{align}
T_2(n,d) &= d \Psi(\mathbf{m}^*) + \frac{\sqrt{d}}{n} \Big[ \max_{\underline{\sigma} \in \Sigma(\mathbf{m}^*)} H_f(\underline{\sigma})  \Big] \nonumber \\
|T_2(n,d)| &\leq d \| f \|_{\infty} + \sqrt{d} \Big| \frac{1}{n} \max_{\underline{\sigma} \in \Sigma(\mathbf{m^*})} H_f(\underline{\sigma}) \Big|. \nonumber
\end{align}

Therefore, 
\begin{align}
\E[T_2(n,d)^2] \leq 2 \Big[ d^2 \| f \|_{\infty}^2 + \frac{d}{n^2} \E\Big[ \Big( \max_{\underline{\sigma} \in \Sigma(\mathbf{m^*})} H_f(\underline{\sigma}) \Big)^2  \Big] \Big] = O(1) \nonumber 
\end{align}
where the final bound follows using Lemma \ref{lemma:gaussian_process_bound}. It remains to bound $\E[ (T_1(n, d))^2]$. To this end, we observe that for all $\mathbf{m} \in {\rm{Sim}}$, 
\begin{align}
d \Psi(\mathbf{m}) + \frac{\sqrt{d}}{n} \Big[ \max_{\underline{\sigma} \in \Sigma(\mathbf{m})} H_f(\underline{\sigma})  \Big] &\leq d \| f \|_{\infty} +  \frac{\sqrt{d}}{n} \Big[ \max_{\underline{\sigma} \in [q]^n} H_f(\underline{\sigma}) \Big]\nonumber \\
&\leq d \|f \|_{\infty} + \frac{\sqrt{d}}{n} \Big|\max_{\underline{\sigma} \in [q]^n} H_f(\underline{\sigma})   \Big|  +  \frac{\sqrt{d}}{n} \Big|\min_{\underline{\sigma} \in [q]^n} H_f(\underline{\sigma})   \Big|.\nonumber
\end{align}

Similarly, we have, 
\begin{align}
&d \Psi(\mathbf{m}) + \frac{\sqrt{d}}{n} \Big[ \max_{\underline{\sigma} \in \Sigma(\mathbf{m})} H_f(\underline{\sigma})  \Big] \geq -d \| f \| _{\infty}  + \frac{\sqrt{d}}{n} \min_{ \underline{\sigma} \in [q]^n } H_f(\underline{\sigma}) \nonumber \\
&\geq -d \| f \| _{\infty} - \frac{\sqrt{d}}{n} \Big|  \min_{ \underline{\sigma} \in [q]^n } H_f(\underline{\sigma}) \Big| - \frac{\sqrt{d}}{n} \Big|  \max_{ \underline{\sigma} \in [q]^n } H_f(\underline{\sigma}) \Big| .\nonumber 
\end{align}

Combining, we have, 
\begin{align}
&-d \| f \| _{\infty} - \frac{\sqrt{d}}{n} \Big|  \min_{ \underline{\sigma} \in [q]^n } H_f(\underline{\sigma}) \Big| - \frac{\sqrt{d}}{n} \Big|  \max_{ \underline{\sigma} \in [q]^n } H_f(\underline{\sigma}) \Big| \nonumber\\
&\leq \max_{\mathbf{m} \in {\rm{Sim}} } \Big[d \Psi(\mathbf{m}) + \frac{\sqrt{d}}{n} \Big[ \max_{\underline{\sigma} \in \Sigma(\mathbf{m})} H_f(\underline{\sigma})  \Big]  \Big] \nonumber \\
 &\leq d \|f \|_{\infty} + \frac{\sqrt{d}}{n} \Big|\max_{\underline{\sigma} \in [q]^n} H_f(\underline{\sigma})   \Big|  +  \frac{\sqrt{d}}{n} \Big|\min_{\underline{\sigma} \in [q]^n} H_f(\underline{\sigma})   \Big| \nonumber 
\end{align}
Thus we obtain the bound on the absolute value 
\begin{align}
 \Big|\max_{\mathbf{m} \in {\rm{Sim}} } \Big[d \Psi(\mathbf{m}) + \frac{\sqrt{d}}{n} \Big[ \max_{\underline{\sigma} \in \Sigma(\mathbf{m})} H_f(\underline{\sigma})  \Big]  \Big| \leq d \|f \|_{\infty} + \frac{\sqrt{d}}{n} \Big|\max_{\underline{\sigma} \in [q]^n} H_f(\underline{\sigma})   \Big|  +  \frac{\sqrt{d}}{n} \Big|\min_{\underline{\sigma} \in [q]^n} H_f(\underline{\sigma})   \Big|. \nonumber 
\end{align}
\noindent
We have, by Cauchy-Schwarz, 
\begin{align}
\E[(T_1(n,d))^2] \leq 3 \Big[ d^2 \|f \|_{\infty}^2 + \frac{d}{n^2} \E\Big[\Big( \max_{\underline{\sigma} \in [q]^n} H_f(\underline{\sigma}) \Big)^2 \Big] +  \frac{d}{n^2} \E\Big[\Big( \min_{\underline{\sigma} \in [q]^n} H_f(\underline{\sigma}) \Big)^2 \Big]  \Big] \nonumber
\end{align}

Finally, we note that $\min_{\underline{\sigma} \in [q]^n} H_f(\underline{\sigma}) \stackrel{d}{=} - \max_{\underline{\sigma} \in [q]^n } H_f(\underline{\sigma})$. This reduces the bound on the second term to 
\begin{align}
\E\Big[ \Big( T_1(n,d) \Big)^2 \Big] \leq 3 \Big[ d^2 \|f \|_{\infty}^2 + \frac{2d}{n^2} \E\Big[\Big( \max_{\underline{\sigma} \in [q]^n} H_f(\underline{\sigma}) \Big)^2 \Big]  \Big]   \nonumber 
\end{align}

The proof is now complete upon using Lemma \ref{lemma:gaussian_process_bound}. }


\section*{Acknowledgements}
The author thanks Prof Amir Dembo for suggesting this problem, numerous helpful discussions, and constant encouragement. The author also 
thanks Prof Andrea Montanari and Prof Sourav Chatterjee for many helpful comments, which improved the presentation of the paper. The author thanks Prof Federico Ricci-Tersenghi for pointing out \cite{leuzzi2001sat} to him. The author also thanks the anonymous referee for a careful reading of the manuscript, and for pointing out some errors in the initial version of the paper.

\bibliographystyle{amsalpha}

\bibliography{graphref}

\end{document}